\documentclass[11pt]{amsart}

\usepackage{amsmath,amsthm,amssymb,mathtools}
\usepackage{fourier}
\usepackage{hyperref}
\usepackage{bm}
\usepackage{xcolor}
\usepackage{parskip}
\usepackage{cite}
\usepackage{comment}
\usepackage{setspace}

\usepackage{lipsum}

\DeclareMathAlphabet{\mathcal}{OMS}{cmsy}{m}{n}

\hypersetup{
  colorlinks=true,
  linktoc=all,
  linkcolor=blue,
  citecolor=red,
  urlcolor=blue,
}

\makeatletter
\renewcommand\subsection{\@startsection{subsection}{2}{0pt}%
{1ex}
  {0.5ex}
  {\normalfont\normalsize\bfseries}}
\makeatother

\numberwithin{equation}{section}
\allowdisplaybreaks
\hypersetup{unicode, psdextra}
\newcommand{\noopsort}[1]{}

\theoremstyle{plain}
\newtheorem{theorem}{Theorem}[section]
\newtheorem{proposition}[theorem]{Proposition}
\newtheorem{lemma}[theorem]{Lemma}

\newtheorem*{conjecture*}{Conjecture}
\newtheorem{claim}[theorem]{Claim}

\theoremstyle{definition}
\newtheorem{definition}[theorem]{Definition}

\theoremstyle{remark}
\newtheorem{remark}[theorem]{Remark}

\newcommand{\R}{\mathbb{R}}

\newcommand{\e}{\varepsilon}
\newcommand{\calJ}{\mathcal{J}}
\newcommand{\bt}{\mathbf{t}}
\newcommand{\calI}{\mathcal{I}}
\newcommand{\calT}{\mathcal{T}}
\newcommand{\calZ}{\mathcal{Z}}
\newcommand{\calL}{\mathcal{L}}
\newcommand{\calB}{\mathcal{B}}
\newcommand{\calS}{\mathcal{S}}
\newcommand{\calg}{\mathcal{G}}

\newcommand{\bzero}{\mathbf{0}}
\newcommand{\Int}{\mathrm{Int}}
\newcommand{\supp}{\mathrm{supp}}

\newcommand{\rank}{\mathrm{rank}}
\newcommand{\Ker}{\mathrm{Ker}}
\newcommand{\corank}{\mathrm{corank}}

\title{On {H}ausdorff dimensions of $k$-point configuration sets and Elekes-R\'onyai type theorems}
\author{Minh-Quy Pham}
\address{Department of Mathematics, University of Rochester, Rochester, NY 14627.} \email{qpham3@ur.rochester.edu, quypham.math@gmail.com}
\date{\today}

\begin{document}
\sloppy
\thispagestyle{empty}
\date{\today}
\begin{abstract} \begin{spacing}{1.}
We prove a ``dimension expansion'' version of the Elekes-R\'onyai theorem for trivariate real analytic functions: If $f$ is a trivariate real analytic function, 
then $f$ is either locally of the form $g(h(x)+k(y)+l(z))$, or the following is true: 
whenever a Borel set $A\subset\R$ has Hausdorff dimension $\alpha\in \left(\frac{1}{2},1\right)$,
$f(A\times A\times A)$ has dimension significantly larger than that of $A$, i.e.
\begin{align*}
    \dim_Hf(A\times A\times A)\geq \alpha+\varepsilon(\alpha),\quad \text{for some } \varepsilon(\alpha)>0,
\end{align*}
Moreover, if $\alpha>\frac{2}{3}$, $f(A\times A\times A)$ has positive Lebesgue measure. This is a considerable extension of the result established by Koh, T. Pham, and Shen (\textit{J. Funct. Anal.} \textbf{286} (2024)).

We also obtain an alternative proof and an improvement for the Elekes-R\'onyai type theorem for bivariate real analytic functions established by Raz and Zahl (\textit{Geom. Funct. Anal.} \textbf{34} (2024)).

We derive these from more general results, 
showing that various $k$-point configuration sets of thin sets have positive Lebesgue measure by exploiting the 
optimal
$L^2$-based Sobolev estimates for the associated family of Fourier integral operators.
Extending the framework developed by Greenleaf, Iosevich, and Taylor 
(\textit{Mathematika} \textbf{68} (2022),  \textit{Math. Z.} \textbf{306} (2024)) to prove Mattila-Sj\"olin type theorems,
we obtain Falconer-type results for many configuration sets on which the method would be vacuous if demanding nonempty interior.
In particular, when $k=2$, we generalize the Falconer-type result for metric functions in $\R^d$ satisfying strong non-vanishing curvature conditions established by 
Eswarathasan, Iosevich, and Taylor (\textit{Adv. Math.} \textbf{228} (2011)) 
and the asymmetric Mattila-Sj\"olin type results of Greenleaf, Iosevich, and Taylor (\textit{J. Geom. Anal.} \textbf{31} (2021)) to a broader class of smooth functions of asymmetric form. $\Phi: \R^m\times \R^n\to \R$, for $1\leq m\leq n$, under certain mild nondegeneracy conditions.

\vspace{0.2cm}
\noindent \textbf{Keywords.} Elekes-R\'onyai theorem, Fourier integral operators, Hausdorff dimension.

\noindent \textbf{MSC.} primary: 28A75, 28A80, 58J40, 42B10; secondary: 52C10.
\end{spacing}
\end{abstract}

\maketitle


\section{Introduction}\label{sec:introduction}
In this article, we use {Fourier integral operators} techniques to study two related sets of problems in geometric measure theory:  Elekes-R\'onyai 
type theorems on dimension expanding properties of polynomials, and expansion, Falconer-type and Mattila-Sj\"olin type problems for  $k$-point $\Phi$-configurations.

In connection with the Elekes-R\'onyai theorem on expanding polynomials \cite{Elekes_Ronyai_2000,Raz_Zahl_2024}, we discover that if a bivariate real analytic 
function does not have a specific form (for example $f(x,y)=x+y$), then the related family of Fourier integral operators is associated to a \textit{folding} canonical relation 
(i.e., 
the cotangent spaces projections have at most Whitney fold singularities). We prove that for a Borel subset $A\subset \R$ with large enough Hausdorff 
dimension,
the image set $f(A\times A)$ has dimension equal to $1$ (as large as possible) or even has positive Lebesgue measure. A similar ``dimension expansion'' phenomenon 
also holds for trivariate analytic functions. 

We provide nondegeneracy conditions under which certain types of configuration sets have positive Lebesgue measure or have nonempty interior when the underlying sets have 
sufficiently high Hausdorff dimension. This extends the  methods and results in 
\cite{ESWARATHASAN2011,GreenleafIosevichTaylor2021,GreenleafIosevichTaylor22,GreenleafIosevichTaylor2024} to a wider range of $k$-point configurations. 

\subsection{Dimension expansion and Elekes-R\'onyai type theorems for analytic functions}
The Erd\H{o}s distinct distances problem and other geometric questions have strong connections with the expanding polynomials problem in additive combinatorics 
and discrete geometry; and many combinatorial questions about algebraic objects often lead to applications to geometric problems such as distances, lines, circles, 
sum-product type problems, and so on. In 1983, Erd\H{o}s and Szemer\'edi \cite{Erdos_Szemeredi_1983} proved that if $A\subset \R$ is a finite set, then either the 
sum set $A+A=\{a+a': a,a'\in A\}$ or the product set $A\cdot A=\{aa': a,a'\in A\}$ must have cardinality much larger than of $A$, that is, there exists $\e>0$ such that
\begin{align*}
    \#(A+A)+\#(A\cdot A) \gtrsim (\#A)^{1+\e}.
\end{align*}
On the one hand, the intuition for this result is that it is not possible for a large set to behave 
simultaneously like both an arithmetic and a geometric progression.
On the other hand, this result also hinted that since polynomials are combinations of additions and multiplications, one can expect that a similar result 
holds true for generic polynomials, unless the polynomial has a group-related special form.

The study of algebraic structures of zero sets of polynomials and their intersections with Cartesian products originated from the work of 
Elekes and R\'onyai \cite{Elekes_Ronyai_2000}, and Elekes and Szab\'o \cite{Elekes_Szabo_2012}, which had many applications to Erd\H{o}s type problems. 
Elekes and R\'onyai  proved that if $f$ is a real bivariate polynomial, then either $f$ is one of the special forms $g(h(x)+k(y))$ 
or $g(h(x)\cdot k(y))$, where $g,h,k$ are univariate real polynomials, or that for all finite sets $A,B\subset \R$ of cardinality $n$ sufficiently large, we have
\begin{align*}
    \# f(A\times B) \geq cn,
\end{align*}
for some constant $c>0$ that depends on $f$. Recently, the lower bound was improved by Solymosi and Zahl \cite{Solymosi_Zahl_2024} to 
\begin{align*}
    \# f(A\times B)\gtrsim n^{3/2}.
\end{align*}
If $f$ is not a polynomial special form as above, then we call it an expanding polynomial. Generalizations and analogues of the Elekes-R\'onyai theorem in 
combinatorial geometry can be found in 
\cite{Elekes_Szabo_2012, Raz_Sharir_Solymosi_2014,Raz_Sharir_Solymosi_2016,Raz_Sharir_deZeeuw_2018,Raz_Shem-Tov_2020} and the references therein. 
In \cite{Tao_expanding_poly_2015}, Tao established a variant of the Elekes-R\'onyai result for bivariate polynomials over finite fields. For a comprehensive review of 
recent progress and applications of Elekes-R\'onyai type problems, we refer the reader to the survey \cite{deZeeuw_survey_ER_2018}.

Similar questions regarding expanding functions can be asked in the continuous setting, where finite sets of cardinality $n$ can be replaced by fractal sets of 
Hausdorff dimension $\alpha$. Bourgain \cite{Bourgain_2010} proved the following ``dimension expansion'' result for subsets of $\R$:
\begin{theorem}[Bourgain \cite{Bourgain_2010}]\label{thm: Bourgain_2010}
    For each $0<\alpha<1$, there is a number $\e=\e(\alpha)>0$ so that the following holds. Let $A\subset \R$ be a Borel set with $\dim_HA=\alpha$. Then there exists $\lambda\in A$ so that
    \begin{align*}
        \dim_H(A+\lambda A)\geq \alpha+\e.
    \end{align*}
\end{theorem}
In fact, Bourgain deduced Theorem \ref{thm: Bourgain_2010} from his deeper results on projections in the discretized setting of $(\delta,\alpha)_2$ sets \cite{Bourgain_2003,Bourgain_2010}. 
His works were influential and led to resolutions and developments on several longstanding questions in geometric measure theory and harmonic analysis 
such as the Kakeya, Furstenberg set, discretized sum-product, discretized projections, and Falconer's distance problem; 
for recent progress, see for example, \cite{ORegan_Shmerkin_Wang_2025,Orponen_Shmerkin_2023,Ren_Wang_2023,Shmerkin_2023,Shmerkin_Wang_2025,Wang_Zahl_Kakeya_2025} 
and the references therein.

In a similar theme with Theorem \ref{thm: Bourgain_2010}, Raz and Zahl \cite{Raz_Zahl_2024} proved a ``dimension expansion'' version of the Elekes-R\'onyai 
theorem for bivariate real analytic functions. 
For convenience,  fix a connected open set $\Omega\subset \R^2$ that contains $[0,1]^2$; given an analytic function $f\in C^\omega(\Omega)$, and  $A,B\subset \R$, we denote $\Delta_f(A,B):=f(A\times B)$.
We recall the following definition and theorem from \cite{Raz_Zahl_2024}. 
\begin{definition}
    Let $f\in C^\omega(\Omega)$ be a real bivariate real analytic function. We say that $f$ is an \textit{analytic special form} if on each connected region of $\Omega\setminus \left(\left\{\partial_xf=0\right\}\bigcup\left\{\partial_yf=0\right\}\right)$, there are univariate real analytic functions $g,h$, and $k$ such that $f(x,y)=g(h(x)+k(y))$.
\end{definition}

\begin{theorem}[Raz and Zahl\cite{Raz_Zahl_2024}]\label{thm: Raz_Zahl_bivariate_analytic}
For every $0<\alpha<1$, there exists $\e=\e(\alpha)>0$ so that the following holds. Let $f\in C^\omega(\Omega)$ be analytic. 
Then either $f$ is an analytic special form or, for every pair of Borel sets $A,B\subset [0,1]$ of Hausdorff dimension at least $\alpha$, we have
\begin{align*}
    \dim_H\Delta_f(A,B)\geq \alpha+\e.
\end{align*}
\end{theorem}
In fact, Raz and Zahl also established a discretized version of the Elekes-Rónyai theorem, which was deduced from an energy dispersion estimate in the setting of $(\delta,\alpha)_2$ sets. This was done by introducing an auxiliary function $\kappa_f(x,y)$, which helps to quantify whether the function $f$ looks like a special form. When the function $\kappa_f(x,y)$ is bounded away from $0$, the authors used ideas from additive combinatorics and applied Shmerkin's nonlinear discretized projection theorem \cite{Shmerkin_2023} to obtain the energy dispersion estimate. They also proved other related results, including on
the pinned distance problem and nonlinear projections involving Blaschke curvature, see \cite{Raz_Zahl_2024} for further details. 

In this paper, we give an alternative proof for Theorem \ref{thm: Raz_Zahl_bivariate_analytic} with explicit lower bounds whenever the dimensions of the underlying sets are sufficiently large. More precisely, we improve the lower bound $\dim_H\Delta_f(A,B)\geq \alpha+\e$ (with $\e$ very small and non-explicit) to $\e(\alpha)=\frac{3\alpha-2}{3}$ when the dimension $\alpha>\frac{2}{3}$. Moreover, we are able to strengthen the conclusion of the theorem 
to $\Delta_f(A,B)$ having positive Lebesgue measure if $\alpha>\frac{5}{6}$.

Our result is formulated as follows.
\begin{theorem}\label{thm: Elekes_Ronyai_analytic_2vars}
Let $\alpha,\beta\in (0,1]$, and let $f\in C^\omega(\Omega)$ be a bivariate real analytic function. Then either $f$ is an analytic special form or, for every pair of Borel sets $A, B\subset [0,1]$ with $\dim_HA=\alpha$, $\dim_HB=\beta$, we have the following:
    \begin{itemize}
        \item[(i)] If $\alpha+\beta\leq \frac{5}{3}$, then $\dim_H\Delta_f(A,B)\geq \max\left\{\alpha+\beta-\frac{2}{3},0\right\}.$
        \item[(ii)] If $\alpha+\beta> \frac{5}{3}$, then $\calL^1(\Delta_f(A,B))>0$. 
    \end{itemize}
\end{theorem}
\begin{remark}
    Although the framework in \cite{Raz_Zahl_2024} is elegant and able to obtain a discretized version, it is effective only for sets of small dimension, producing a dimensional gain $\e(\alpha)$ that is nonexplicit and provides no quantitative information as $\alpha$ approaches $1$. In particular, their method cannot yield that the image set $f(A,B)$ has positive Lebesgue measure. Theorem \ref{thm: Elekes_Ronyai_analytic_2vars} addresses this gap for sets of sufficiently large dimension. From Theorem \ref{thm: Elekes_Ronyai_analytic_2vars} (i), if $\alpha\in \left(\frac{2}{3},\frac{5}{6}\right]$, $A,B\subset [0,1]$ are Borel sets with dimension at least $\alpha$, we obtain an explicit dimensional estimate
   \[\dim_H\Delta_f(A,B)\,\geq \alpha+\left(\alpha-\frac{2}{3}\right),\] 
   which improves Theorem \ref{thm: Raz_Zahl_bivariate_analytic} in this range. When $\alpha\in \left(\frac{5}{6},1\right]$, we obtain   $\mathcal{L}^1\left(\Delta_f(A,B)\right)>0$ by establishing the $L^2$-integrability of the configuration measure, a qualitatively stronger conclusion that cannot be deduced from the dimensional gain alone.
   This result is new and does not appear to be obtained by other methods.
\end{remark}
Theorem \ref{thm: Elekes_Ronyai_analytic_2vars}, as well as the trivariate Theorem \ref{thm: ElekesRonyai_functions_3vars_analytic_general}, 
can be derived from our general theorems, for which we use \textit{Fourier integral operators} (FIO) and microlocal analysis techniques, extending the applicability of the partition optimization approach of \cite{GreenleafIosevichTaylor2021,GreenleafIosevichTaylor22,GreenleafIosevichTaylor2024} from Mattila-Sj\"olin type theorems (nonempty interior) to Falconer-type (positive Lebesgue measure) and expansion-type. This method is very different from the method used in Raz and Zahl's work, but we shall see that the aforementioned auxiliary function $\kappa_f(x,y)$ also appears naturally in our context. When $\kappa_f(x,y)$ is nonzero, we show that the relevant family of FIOs is associated to a \textit{folding canonical relation} $\Lambda$
in the sense of Melrose and Taylor \cite{Melrose_Taylor_1985},
 that is, $\Lambda$ exhibits \textit{Whitney fold singularities} in both the left and right cotangent space projections $\pi_L$ and $\pi_R$. The nature of the singularities allows us to exploit the sharp $L^2$-based Sobolev estimates  \cite{Melrose_Taylor_1985} for these operators. This provides us a way to bypass the technically intricate discretized arguments and the use of nonlinear projection theorems in \cite{Raz_Zahl_2024}. The details will be discussed in later sections.
\begin{remark}
    Recently, Demeter and O'Regan \cite{Demeter_O'Regan_2026} considered a particular expanding function $f(x,y)=x(x+y)$ and obtained an improved bound for this function. Given Borel sets $A,B\subset [0,1]$ with $\dim_HA=\alpha$, $\dim_HB=\beta$, where $\alpha,\beta\in \left(0,\frac{2}{3}\right]$, they showed that
    \begin{align*}
        \dim_H\Delta_f(A,B)\geq \frac{2(\alpha+\beta)}{3}.
    \end{align*}
    This result was derived from a $\delta$-discretized version in the setting of Katz and Tao $(\delta,\alpha)_2$ sets, see \cite{Demeter_O'Regan_2026} for further details. However, this result does not yield a range of $\alpha,\beta$ such that $\Delta_f(A,B)$ has positive Lebesgue measure.
\end{remark}
The Elekes-R\'onyai theorem for polynomials of three variables was obtained in \cite{Schwartz_Solymosi_deZeeuw_2013,Raz_Sharir_deZeeuw_2018}, so it might be expected that a result similar to Theorem \ref{thm: Raz_Zahl_bivariate_analytic} and Theorem \ref{thm: Elekes_Ronyai_analytic_2vars} might hold for trivariate analytic functions. In this direction, assuming that $f(x,y,z)$ is a quadratic polynomial of three variables, depends on each variable, and does not have the form $g(h(x)+k(y)+l(z))$, Koh, T. Pham and Shen \cite{Koh_Pham_Shen_2024} proved that if $A,B, C\subset \R$ are Borel sets satisfying $\dim_H A+\dim_H B+\dim_HC>2$, then the set
\begin{align*}
    \Delta_f(A,B,C):=f(A\times B\times C)=\left\{f(x,y,z): x\in A,y\in B,z\in C\right\}
\end{align*}
has positive Lebesgue measure. The exclusion of polynomials of the form $g(h(x)+k(y)+l(z))$ is natural and necessary,
and thus we have the following definition of analytic functions of special forms.
Keeping the notation of the bivariate case, now let $\Omega\subset\R^3$ be a connected open set that contains $[0,1]^3$, and for an analytic function 
$f: \Omega\to \R$ and Borel sets $A,B,C\subset [0,1]$, denote $\Delta_f(A,B,C):=f(A\times B\times C)$.
\begin{definition}\label{def: analytic_special_form_3_vars}
    Let $f\in C^\omega (\Omega,\R)$ be an analytic function on 
    $\Omega$. We say that $f$ is an \textit{analytic special form} if on each connected region of
   \[\Omega\setminus\left(\left\{\partial_{x}f=0\right\}\cup \left\{\partial_{y}f=0\right\}\cup \left\{\partial_{z}f=0\right\}\right),\] there are univariate real analytic functions $g,h,k,$ and $l$ such that $f(x,y,z)=g(h(x)+k(y)+l(z))$.
\end{definition}

The following Elekes-R\'onyai type theorem in three variables  is a considerable extension of Koh, T. Pham, and Shen \cite{Koh_Pham_Shen_2024}.
\begin{theorem}\label{thm: ElekesRonyai_functions_3vars_analytic_general}
    Let $\alpha,\beta,\gamma\in (0,1]$. Let $f\in C^\omega(\Omega)$ be a trivariate real analytic function such that $\partial_xf$, $\partial_yf$, and $\partial_zf$ do not vanish identically on $\Omega$. Then either $f$ is an analytic special form, or for Borel sets $A,B,C\subset [0,1]$ with $\dim_HA=\alpha$, $\dim_HB=\beta$, and $\dim_HC=\gamma$, we have the following:
    \begin{itemize}
        \item[(i)] If $\alpha+\beta+\gamma\leq 2$, then $\dim_H\Delta_f(A,B,C)\geq \max\{\alpha+\beta+\gamma-1,0\}$.
        \item[(ii)] If $\alpha+\beta+\gamma>2$, then $\calL^1\left(\Delta_f(A,B,C)\right)>0$.
    \end{itemize}
\end{theorem} 
The proof of Theorem \ref{thm: ElekesRonyai_functions_3vars_analytic_general} follows along the same lines as 
Theorem \ref{thm: Elekes_Ronyai_analytic_2vars}, although technically more involved. We introduce a collection of auxiliary functions $\left\{\mathcal{G}_i(x,y,z)\right\}_{i=1}^3$, which helps to measure whether the function $f$ looks like a special form. A key fact is that for an analytic function $f: \Omega\to \R$, these auxiliary functions vanish identically on $\Omega$ if and only if $f$ is locally of special form. On the other hand, if at least one of these functions is nonzero, we show that the relevant family of FIOs is associated to a \textit{nondegenerate canonical relation}, allowing the use of the 
known optimal
$L^2$-based Sobolev estimates to obtain the results.
\begin{remark}
    \begin{itemize}
        \item[] 
        \item[(i)] Theorem \ref{thm: ElekesRonyai_functions_3vars_analytic_general} (ii) recovers the result by Koh, T. Pham, and Shen \cite{Koh_Pham_Shen_2024} on quadratic polynomials. In that paper, they adapted a combinatorial argument from the finite field model with the FIO framework introduced by Eswarathasan, Iosevich, and Taylor \cite{ESWARATHASAN2011} to prove the result; however, the proof was technically intricate and highly specialized to the quadratic setting of three variables, making generalization highly nontrivial. In \cite{Pham_Falconerfunction_2025}, we gave a short proof of this result using a projection theoretic approach. Recently, in \cite{Liao_Pham_Shen_2026}, Liao, T. Pham, and Shen proved an estimate for the smoothed $L^2$ energy  for the natural measure supported on $\Delta_f(A,B,C)$.
        \item[(ii)] We note that a finite field Elekes-R\'onyai type theorem for quadratic polynomials of three variables was established in \cite{Pham_Vinh_deZeeuw_2019}. Some recent progress on expanding polynomial problems in the finite field setting can be found in \cite{AralaChow24}.
        \end{itemize}
\end{remark}
\begin{remark}
    \begin{itemize}
        \item[]
        \item[(i)] The dimensional threshold in Theorem \ref{thm: ElekesRonyai_functions_3vars_analytic_general} (ii) is sharp. For example, take $f(x,y,z)=x(y+z)$, $A=\left\{ 0\right\}$, $B=C=[0,1]$. Then
    \begin{align*}
        \dim_HA+\dim_HB+\dim_HC=2,
    \end{align*}
    but $\dim_H\Delta_f(A,B,C)=0$, hence $\calL^1\left(\Delta_f(A,B,C)\right)=0$. This example also shows that conditions $\dim_HA, \dim_HB,\dim_HC>0$ are necessary. 
    \item[(ii)] In Theorem \ref{thm: ElekesRonyai_functions_3vars_analytic_general}, the conditions that $f$ must depend on all variables and is not an analytic special form are necessary. Indeed, for example, if $\partial_xf\equiv 0$, then $f$ depends only on two variables $y$ and $z$. \\ On the other hand, if $f(x,y,z)=x+y+z$, then for any $0\leq\alpha\leq \beta\leq \gamma\leq1$, one can find a compact set $A\subset [0,1]$ such that 
    \begin{align*}
        \dim_HA=\alpha, \quad \dim_H(A+A)=\beta,\quad\text{ and }\quad \dim_H(A+A+A)=\gamma,
    \end{align*}
    This follows from a result of Schmeling and Shmerkin \cite{Schmeling_Shmerkin_2010} (see also \cite{Falconer85}).
    \end{itemize}
\end{remark}
\begin{remark}
    We also extend the dimension expansion results to smooth functions. 
    The precise statements of these results can be found in Sections \ref{sec: k_point_configurations}, \ref{sec: dimexpansion_2vars} and \ref{sec: dimexpansion_3vars}. 
\end{remark}
\begin{remark}
It would clearly be desirable to extend these results to $k$-variate functions $f$, and, as described below, the underlying 
microlocal approach is formulated in this setting.
While it may be possible to characterize analytic special forms for $k$-variate real analytic functions using our method,
we do not pursue that goal in the present paper, 
due to the fact that the analysis of the canonical relations and FIOs would be much more complicated. 
We hope to return to this in future work.
    \end{remark}

\subsection{\texorpdfstring{$k$}{k}-point configuration sets}
We extend our study to Falconer-type problems for various \textit{$k$-point configuration sets} using Fourier integral operator techniques, showing that certain types of configuration sets have positive Lebesgue measure when the underlying sets have sufficiently large Hausdorff dimension. We start by recalling the formulation of $\Phi$-configuration sets introduced by Grafakos, Greenleaf, Iosevich, and Palsson in \cite{Grafakosetal15}. 

Let $k,p\geq 1$, and let $X^i$, be a smooth manifold of dimension $d_i$, for each index $1\leq i \leq k$. Denote $X=X^1\times \dots \times X^k$ and put $\dim X=\sum\limits_{i=1}^kd_i=n$.

\begin{definition}
    Let $\Phi\in C^\infty(X,\R^p)$ be a smooth function. Suppose that $A_i\subset X^i$, $1\leq i\leq k$, are compact sets. Then the \textit{$k$-point $\Phi$-configuration set} of the $A_i$ is
\begin{align*}
    \Delta_\Phi(A_1,\dots,A_k)\coloneqq \left\{  \Phi(x^1,\dots,x^k): x^i\in A_i,1\leq i\leq k\right\}\subset \R^p.
\end{align*}
\end{definition}
The main goal is to find nondegeneracy conditions on $\Phi$, together with a dimensional threshold, $s_0=s_0(\Phi)\in (0,\dim X)$, such that whenever $\sum\limits_{i=1}^k\dim_H(A_i)>s_0$, the configuration set $\Delta_\Phi(A_1,\dots,A_k)$ has positive $p$-dimensional Lebesgue measure. We note that the critical exponent $s_0(\Phi)$ may not exist for general smooth functions without any further assumptions. Indeed, due to Whitney's extension theorem \cite{Whitney_extension_1934}, given any closed set $K\subset \R^n$, one can construct a smooth function $f:\R^n\to [0,\infty)$ such that $K=f^{-1}(0)$, hence $\calL^1(f(K))=0$. Note that $K$ can be a fractal set of large Hausdorff dimension. Thus, in order to obtain Falconer-type results, one needs to impose certain nondegeneracy conditions on configuration functions. 

For simplicity, we shall first discuss the case where $k=2$ and $p=1$; later on we shall discuss general cases. We shall state our results for compact sets, noting that the results naturally extend to the Borel case.

In 1985 \cite{Falconer85}, Falconer introduced the distance set problem, showing that if a compact set $A\subset\R^d$, $d\geq 2$, has Hausdorff dimension $\dim_H A>\frac{d+1}{2}$, then its
\textit{distance set}
\begin{align*}
    \Delta(A):=\left\{\, \vert x-y\vert: x,y\in A\right\}
\end{align*}
has positive Lebesgue measure. In addition, he conjectured that the dimensional threshold should be $\frac{d}{2}$, for $d\geq 2$. This conjecture remains open in all dimensions, and the exponent $\frac{d}{2}$ is known to be sharp. In recent decades, a variety of analytic and geometric techniques have been developed to lower the dimensional threshold that is sufficient for the distance set $\Delta(A)$ to have positive Lebesgue measure. The best currently known thresholds are given by
\begin{align*}
  \dim_HA>  \left\{ 
\begin{array}{lll}
    \tfrac{5}{4}, & d=2& (\text{Guth, Iosevich, Ou, and Wang \cite{Guth_Iosevich_Ou_Wang_2019}})\\
    \frac{d}{2}+\frac{1}{4}-\frac{1}{8d+4},&d\geq 3& (\text{Du, Ou, Ren, and Zhang \cite{Duetal23a,Duetal23b}})
\end{array}
   \right..
\end{align*}
Earlier improvements on the dimensional thresholds in the distance problem can be found in \cite{Bourgain_distance_1994,Duetal21a, Duetal21b, DuZhang19, Erdogan05, Wolff99}. 

The first general result for $2-$point  $\Phi$-configuration functions was established by Eswarathasan, Iosevich, and Taylor in \cite{ESWARATHASAN2011}, where they showed that if $\Phi: \R^d\times \R^d\to \R$, $d\geq 2$, has nonzero \textit{Monge-Amp\`ere determinant} (also called \textit{Phong-Stein rotational curvature condition}), i.e.,
     \begin{align}\label{eq: Phong-Stein_intro}
        \det \begin{bmatrix}
        0 &\nabla_y\Phi\\
        (\nabla_x\Phi)^T& \frac{\partial^2\Phi}{\partial x_i\partial y_j}
        \end{bmatrix}\neq 0,
    \end{align}     
     and $A,B\subset \mathbb{R}^d$ are compact sets such that $\dim_H A+\dim_HB>d+1$, then the image set\[\Delta_\Phi(A,B);=\left\{\Phi(x,y): x\in A,y\in B\right\}\] has positive Lebesgue measure. For the distance set problem, $\Phi$ is the Euclidean distance function  $\Phi(x,y):=|x-y|$, for $x,y\in \R^d$, and $\Delta_\Phi(A,B)=\Delta(A,B)$ is the distance set. 
     
    In general, we consider a smooth function $\Phi:X\times Y\to \R$, where $X\subset \R^{d_X}$, $Y\subset \R^{d_{Y}}$ are open sets, for $d_X,d_Y\geq 1$. The smooth function $\Phi$ neither necessarily satisfies the nonvanishing curvature conditions nor defines a metric ($d_X\neq d_Y$). However, one may expect that functions satisfying certain mild nondegeneracy conditions still admit Falconer-type results, even in the case where each variable contributes asymmetrically. Moreover, it is natural to expect that the threshold $s_0(\Phi)$ is determined by the degree of degeneracy of the function, i.e., the higher order of degeneracy results in a larger threshold $s_0(\Phi)$. In fact, this relationship can be characterized by the rank of the mixed Hessian of the configuration function.

Our main result in this setting reads as follows.
    \begin{theorem}\label{thm: 2_point_explicit_p=1_rank_condition_intro}
    Let $X\subset \R^{d_X}$, $Y\subset \R^{d_{Y}}$ be open sets, where $d_X,d_Y\geq 1$. Let $\Phi\in C^\infty(X\times Y,\R)$ be a smooth function. Assume that 
    \begin{align*}
        \nabla _x\Phi(x,y)\neq \bzero, \nabla_{y}\Phi(x,y)\neq \bzero, \quad \forall (x,y)\in X\times Y,
    \end{align*}
    and for some $r\geq 1$, 
    \begin{align*}
        \rank\left[\frac{\partial^2\Phi}{\partial x_i\partial y_j}\right]\geq r, \quad \forall (x,y)\in X\times Y.
    \end{align*}
    Assume that $A\subset X$ and $B\subset Y$ are compact sets. Then the following hold:
    \begin{itemize}
         \item[(i)] For any  $\, 0<u<1$, if $\dim_HA+\dim_HB>d_X+d_Y-r+u$,
then $\dim_H(\Delta_\Phi(A,B))\geq u.$
    \item[(ii)] If
    $\dim_HA+\dim_HB>d_X+d_Y+1-r$,
then $\calL^1(\Delta_\Phi(A,B))>0$.
\item[(iii)] If  
    $\dim_HA+\dim_HB>d_X+d_Y+2-r$,
then $\Delta_\Phi(A,B))$ has nonempty interior.
    \end{itemize}
\end{theorem}
It is easy to see that if $\Phi: \R^d\times \R^d\to \R$ satisfies the nonvanishing curvature condition \eqref{eq: Phong-Stein_intro}, then $\rank \left[  \frac{\partial^2\Phi}{\partial x_i\partial y_j}\right]=d$ everywhere. An immediate corollary of Theorem \ref{thm: 2_point_explicit_p=1_rank_condition_intro} is the following.
\begin{theorem}\label{thm: 2_configuration_Phong_Stein_intro}
        Let $X,Y\subset \R^d$ be open sets, for $d\geq 1$. Let $\Phi: X\times Y\to \R$ be a smooth function satisfying the Phong-Stein rotational curvature condition \eqref{eq: Phong-Stein_intro}. Assume that $A\subset X$ and $B\subset Y$ are compact sets. Then the following hold:
        \begin{itemize}
\item[(i)] For any $0<u<1$, if $\dim_HA+\dim_HB>d+u$,
then $\dim_H\left(\Delta_\Phi(A,B)\right)\geq u.$
    \item[(ii)] If
    $\dim_HA+\dim_HB>d+1$,
then $\calL^1\left(\Delta_\Phi(A,B)\right)>0$.
\item[(iii)] If  
    $\dim_HA+\dim_HB>d+2$,
then $\Delta_\Phi(A,B))$ has nonempty interior.
\end{itemize}
    \end{theorem}
We note that Theorem \ref{thm: 2_configuration_Phong_Stein_intro} (ii) recovers the result by Eswarathasan, Iosevich, and Taylor \cite{ESWARATHASAN2011} where they used an FIO approach to extend Falconer's theorem to the class of smooth metric functions satisfying the curvature condition \eqref{eq: Phong-Stein_intro}. 
Parts (i) and (iii) of Theorem \ref{thm: 2_configuration_Phong_Stein_intro} provide us with more information about the Hausdorff dimension of the configuration set. In \cite{Iosevich_Liu_pinned_distance_2019}, Iosevich-Liu proved a stronger result for pinned $\Phi$-configuration sets if $\Phi$ further satisfies the Sogge’s cinematic curvature condition. However, the results in \cite{ESWARATHASAN2011} and \cite{Iosevich_Liu_pinned_distance_2019} do not apply to functions of the asymmetric form $\Phi:\mathbb{R}^{d_X}\times \mathbb{R}^{d_Y}\to \mathbb{R}$, for $d_X\neq d_Y$.

\begin{remark} 
    \begin{itemize}
    \item[] 
    \item[(i)] Theorem \ref{thm: 2_point_explicit_p=1_rank_condition_intro} extends the study initiated in \cite{ESWARATHASAN2011} to a wider class of smooth functions $\Phi: X\times Y\to \R$, where $d_X$ and $d_Y$ are not necessarily equal to each other. 
    Results of Mattila-Sj\"olin type, i.e., as in part (iii), were already obtained in Greenleaf, Iosevich and Taylor \cite{GreenleafIosevichTaylor2021}.
    The nondegeneracy conditions on $\Phi$ and the resulting $L^2$-Sobolev estimates are standard.
    The dimensional thresholds $s_0(\Phi)$ are explicit, depending on the dimensions of the ambient spaces and the rank of the mixed Hessian of $\Phi$.
    
    In the following observations, we will focus on part (ii) of Theorem \ref{thm: 2_point_explicit_p=1_rank_condition_intro}:
    \item[(ii)] The nondegeneracy conditions on $\Phi$ in Theorem \ref{thm: 2_point_explicit_p=1_rank_condition_intro} are necessary. Indeed, if $\nabla_x\Phi\equiv 0$, or $\nabla_y\Phi\equiv 0$, then $\Phi$ does not depend on $x$ or $y$, respectively. If $\rank \left[  \frac{\partial^2\Phi}{\partial x_i\partial y_j}\right]$ vanishes identically, then $\Phi$ has no mixed term $f(x)g(y)$, and we cannot obtain positive results. This has been shown in Elekes-R\'onyai type theorems, as the special form $f(x,y)=h(x)+k(y)$ has vanishing mixed Hessian. For example, consider $\Phi: \R^{d_X}\times \R^{d_Y}\to \R$,  $\Phi(x,y)= \sum_{i=1}^k(x_i+y_i)$, for $x\in \R^{d_X}$, $y\in \R^{d_Y}$. One can use Jarnik's Theorem (see \cite{Falconer85}) to construct a compact set $C\subset [0,1]$ such that $\dim_HC=1$ and $\calL^1\left(\sum_{i=1}^{d_X+d_Y}C\right)=0$. Putting $A= \underbrace{C\times \dots \times C}_{d_X}\subset \R^{d_X}$ and $B=\underbrace{C\times \dots\times C}_{d_Y}\subset\R^{d_Y}$, then one has $\dim_HA+\dim_HB=d_X+d_Y$, but $\calL^1(\Delta_\Phi(A,B))=0$.
    \item[(iii)] For any smooth function $\Phi: X\times Y\to \R$, we always have $\rank \left[  \frac{\partial^2\Phi}{\partial x_i\partial y_j}\right]\leq \min \left\{  d_X,d_Y\right\}$. Hence, the best dimensional threshold that one can expect from Theorem \ref{thm: 2_point_explicit_p=1_rank_condition_intro} (ii) is
    \begin{align*}
        \dim_HA+\dim_HB>\max\left\{ d_X,d_Y\right\}+1.
    \end{align*}
    This bound is not sharp for many configuration functions of interest, such as distance or dot product problems (when $d_X=d_Y$). However, we shall see later that in certain cases, if $\Phi$ satisfies stronger nondegeneracy conditions, then we can lower this dimensional threshold to
    \begin{align*}
        \dim_HA+\dim_HB> \frac{d_X+d_Y+1}{2}.
    \end{align*}
    \item[(iv)]  In view of Theorem \ref{thm: 2_point_explicit_p=1_rank_condition_intro} (ii), when $\rank \left[  \frac{\partial^2\Phi}{\partial x_i\partial y_j}\right]\leq 1$ everywhere, one has the lower bound
    \begin{align*}
        \dim_HA+\dim_HB>d_X+d_Y,
    \end{align*}
    which is vacuous. Hence Theorem \ref{thm: 2_point_explicit_p=1_rank_condition_intro} (ii) cannot ensure the existence of the critical exponent $s_0(\Phi)\in\left(0,d_X+d_Y\right)$. Again, if $\Phi$ satisfies stronger nondegeneracy conditions, we can obtain positive results even when $\rank (H_{x,y}\Phi)=1$.
    \end{itemize}
\end{remark}
When $k\geq 2$ and $p\geq 1$, it is natural to ask Falconer-type questions for more general geometric configurations. 
These include triangles \cite{Greenleaf_Iosevich_2012,Palsson_Acosta_2023}, simplices \cite{Greenleaf_Iosevich_Liu_Palsson_2015, Iosevich_Pham_Pham_Shen_2025}, $k$-chains \cite{Bennett_Iosevich_Taylor_2016}, volumes \cite{Erdogan_Hart_Iosevich_2013, Greenleaf_Iosevich_Mourgoglou_2015, Shmerkin_Yavicoli_2025}, and graphs encoded by systems of pairwise distances or dot products; see e.g. \cite{Borges_Foster_Ou_Palsson_2025,Chan_Laba_Malabika2016,Iosevich_Liu_pinned_distance_2019,Iosevich_Liu_Xi_2022,Iosevich_Taylor_2019,Ou_Taylor_2022}, and the references therein. 

In a series of three papers, Greenleaf, Iosevich, and Taylor \cite{GreenleafIosevichTaylor2021,GreenleafIosevichTaylor22,GreenleafIosevichTaylor2024} developed a framework to study a general class of $k$-point configurations using FIO techniques. More precisely, they obtained Mattila-Sj{\"o}lin type results for various $\Phi$-configurations, showing that there exists a threshold $r_0(\Phi)>0$ such that if $\sum\limits_{i=1}^k\dim_HA_i>r_0$, then the set $\Delta_\Phi(A_1,\dots,A_k)$ contains an open set, that is, has nonempty interior. These types of questions are motivated by the work of Mattila and Sj{\"o}lin \cite{MattilaSjolin99}, showing that if $A\subset \R^d$, $d\geq 2$, is a compact set with $\dim_HA> \frac{d+1}{2}$, then the distance set $\Delta(A)$ has nonempty interior. 
The results in \cite{GreenleafIosevichTaylor2021,GreenleafIosevichTaylor22,GreenleafIosevichTaylor2024} were obtained by considering the \textit{configuration measure} $\nu$, supported on $\Delta_\Phi(A_1,\dots,A_k)$, and associating to the configuration function $\Phi$ a family of linear or multilinear \textit{generalized Radon transforms}, $\{\mathcal{R}_\bt\}_{\bt\in \R^p}$. The continuity of the density $d\nu$ was derived from the $L^2$-Sobolev mapping properties of this family of generalized Radon transforms, from which they concluded that $\Delta_\Phi(A_1,\dots,A_k)$ has nonempty interior.
One application of the main result in \cite{GreenleafIosevichTaylor2024} was a short proof of Palsson and Romero Acosta's result \cite{Palsson_Acosta_2023} on congruence classes of triangles in $\R^d$. For recent Mattila-Sj{\"o}lin type results for other geometric configurations, see, e.g. \cite{Borges_Foster_Ou_Palsson_2025, Koh_Pham_Shen_2022,GreenleafIosevichTaylor2024,Palsson_Acosta_2025}.

Unfortunately, in some regimes, the configuration functions of interest fail to satisfy the hypotheses of the results in \cite{ESWARATHASAN2011} 
or \cite{GreenleafIosevichTaylor2021,GreenleafIosevichTaylor22,GreenleafIosevichTaylor2024} for the simple reason that 
the sufficient conditions become vacuous, e.g., $\sum\limits_{i=1}^k\dim_HA_i> \sum\limits_{i=1}^kd_i$.
Thus, the results do not provide a nontrivial exponent even for Mattila-Sj\"{o}lin type problems for some classes of functions,
including those of  Elekes-R\'onyai type.
Nevertheless, the techniques developed in those papers 
form the foundation for the work in the current paper. To illustrate this simply, let us consider a smooth function of three variables $f:\R^3\to \R$, and let $A,B,C\subset \R$ be compact sets. Define\[\Delta_f(A,B,C):=f(A\times B\times C)=\left\{ f(x,y,z): x\in A,y\in B,z\in C\right\}.\] 
It is obvious that $f$ is not a metric function, hence the result in \cite{ESWARATHASAN2011} cannot be applied. If one applies Theorem 2.6 in \cite{GreenleafIosevichTaylor2024}, then the set $\Delta_f(A,B,C)$ has nonempty interior if
\begin{align*}
    \dim_HA+\dim_HB+\dim_HC>\max\left\{ 2,1\right\}+1+\beta\geq 3,
\end{align*}
where $\beta\geq 0$ measures the loss (if any) of a related trilinear FIO on $L^2$-Sobolev spaces. 
This is vacuous for all $\beta\geq 0$.

In the present paper, we 
investigate the $L^2$ norm of the configuration measure $\nu(\bt)$ instead of the continuity of its density. Adapting the FIO method developed in \cite{GreenleafIosevichTaylor2021,GreenleafIosevichTaylor22,GreenleafIosevichTaylor2024}, we obtain Falconer-type results for a wide class of $k-$point configurations, which have fascinating applications to the Elekes-R\'onyai problem for analytic functions described above.
 To achieve this goal, the key step is to partition the $2k$ variables into two groups $I,J$, and represent $\int \vert\nu(\bt)\vert^2d\bt$ as a pairing of the tensor product of \textit{Frostman measures} $\mu_i$ supported on some of the $A_i$, $i\in I$, with the value of a Fourier integral operator, $\calT_{IJ}$, acting on the tensor product of Frostman measures $\mu_j$, $j\in J$. This is done for each partition $\{I,J\}$, and from $L^2$-Sobolev mapping properties of these FIOs, we obtain a threshold  $\sum\limits_{i=1}^k\dim_H(A_i)>s_0(I,J)$ guaranteeing that $\int \vert\nu(\bt)\vert^2d\bt$ is finite. This yields that $\nu(\bt)$ is absolutely continuous with respect to Lebesgue measure, and hence $\calL^p(\Delta_\Phi(A_1,\dots A_k)
)>0$. The threshold can be lowered by optimizing over all such partitions \textit{microlocally}. We defer the statements of the general theorems for now and discuss how to apply our general theorems to Falconer-type functions and Elekes-R\'onyai type theorems; see Section \ref{sec: preliminaries} for the necessary background and Section \ref{sec: k_point_configurations} for precise statements of the general results. 

We emphasize that in applications, the local (microlocal) optimization step plays an important role and is highly nontrivial for many configurations of interest. In fact, this step requires a deep understanding of the geometry of the canonical relations associated to configuration functions. This is due to the fact that it is possible that these canonical relations fail to be nondegenerate globally and are only nondegenerate after suitable localizations (microlocalizations) (see \cite{GreenleafIosevichTaylor2024} for various examples). Thus, in general, it would be challenging to give explicit nondegeneracy conditions for configuration functions to have positive result. However, for incidence relations of codimension one, a key observation is that above each point of the incidence relation is just a line, and an open conic cover of the conormal bundle of the incidence relation is equivalent to an open cover of the incidence relation. By localizing the incidence relation carefully, we are able to give explicit nondegeneracy conditions for which functions with codimension one admit positive result. For further details, see Theorem \ref{thm: k_point_explicit_p=1}.

As indicated in the previous sections, Theorem \ref{thm: Elekes_Ronyai_analytic_2vars} and Theorem \ref{thm: ElekesRonyai_functions_3vars_analytic_general} follow from our general results, namely Theorem \ref{thm: general_phi_configurations}. 
To prove Theorem \ref{thm: Elekes_Ronyai_analytic_2vars}, the starting point is to study the $L^2$ norm of the configuration measure $\nu(t)$ supported on $\Delta_f(A,B)$. In the setting of our general theorem, we represent $\int |\nu(t)|^2dt$ by a family of FIOs that is associated to a canonical relation $\Lambda$. Unfortunately, this canonical relation is degenerate and the application of Theorem \ref{thm: k_point_explicit_p=1} is not possible. However, it turns out that $\Lambda$ is degenerate in a specific way. Recall that if $f$ is not an analytic special form, then the auxiliary function $\kappa(x,y)$ does not vanish identically on the domain. 
Our analysis shows that whenever $\kappa(x,y)$ is nonzero, the canonical relation $\Lambda$ is  a folding canonical relation
in the sense of \cite{Melrose_Taylor_1985}:
The cotangent space projections to the left and right, $\pi_L:\Lambda\to T^\ast X_L$ and $\pi_R: \Lambda \to T^\ast X_R$, respectively,
admit  singularities of at worst Whitney fold type. This is shown by carefully chosen parametrizations of the canonical relation $\Lambda$ and using the assumptions that $\kappa(x,y)$ is nonzero. This enables us to exploit the $L^2-$based Sobolev estimates with the loss of $1/6$ derivatives \cite{Melrose_Taylor_1985}, 
from which we can apply Theorem \ref{thm: general_phi_configurations} to obtain the conclusion.

Now we turn to the sketch of the proof of Theorem \ref{thm: ElekesRonyai_functions_3vars_analytic_general}. The idea is the same as in the case of bivariate functions, but we need to introduce a collection of auxiliary functions $\{\calg_i\}_{i=1}^3$, which helps characterize the special forms. A key lemma we need to prove is that these auxiliary functions vanish identically on the domain if and only if $f$ is a special form. To our knowledge, this fact has not been observed before, and we proved this by assuming that the auxiliary functions vanish identically and solving a system of homogeneous first order partial differential equations obtained from this assumption. The solution to this system is exactly a special form. Once we have this lemma, we can assume that $f$ is not a special form, and hence one of the auxiliary functions is nonzero after a suitable localization. Then we show that the related canonical relation $\Lambda$ is nondegenerate, allowing us to use the standard $L^2$ estimates to finally get the conclusion.

The rest of the paper is organized as follows. In Section \ref{sec: preliminaries}, we review the background on FIOs and the related best known $L^2$-Sobolev estimates. We also recall some necessary notions and results regarding Hausdorff dimension of sets, Frostman measures, Sobolev dimensions that we need for our proofs. In Section \ref{sec: k_point_configurations}, we present the details of our FIO framework to study $k-$point configuration problems and prove our main result for general configuration functions, Theorem \ref{thm: general_phi_configurations}, with $k\geq 2$, and $p\geq 1$. When $p=1$, we introduce some explicit nondegeneracy conditions and prove Theorems \ref{thm: 2_point_explicit_p=1_rank_condition_intro} and \ref{thm: 2_configuration_Phong_Stein_intro}. Sections \ref{sec: dimexpansion_2vars} and \ref{sec: dimexpansion_3vars} are devoted to applications of the general results in Section \ref{sec: k_point_configurations} to study the dimension expansion phenomenon of bivariate and trivariate smooth functions, respectively. The proofs of our Elekes-R\'onyai type results, namely Theorems \ref{thm: Elekes_Ronyai_analytic_2vars} and \ref{thm: ElekesRonyai_functions_3vars_analytic_general}, are presented in Sections \ref{sec: dimexpansion_2vars} and \ref{sec: dimexpansion_3vars}, respectively. In section \ref{sec: distance_surface}, we discuss an application of our general result to a variant of the distance set problem, Theorem \ref{thm: distance_surfaces}. Finally, in the last section \ref{sec: final_comments}, we will make some remarks regarding the main results of the paper.

\subsection*{Notation.} Throughout the paper, in any metric space $X$, $B(x,r)$ will stand for the open ball with center $x\in X$ and radius $r>0$. Given a subset $A\subset X$, we denote its interior and closure by $\Int(A)$ and $\overline{A}$, respectively; and $A^c=\R^d\setminus A$ denotes the complement of $A$. The  distance between two nonempty compact sets $A$ and $B$ will be denoted by $d(A,B)$, and between a point $x$ and a compact $A$ by $d(x,A)$. The Hausdorff dimension of $A$ will be denoted by $\dim_H A$, while the dimension of a submanifold (or a real analytic subvariety) $Z\subset \R^d, d\geq 1$, will be denoted by $\dim Z$. We denote by $\calL^d$ the Lebesgue measure on Euclidean space $\mathbb{R}^d$. 
The set of non-zero Radon measures $\mu$ on $\mathbb{R}^d$ with compact support $\supp \mu \subset A$ will be denoted by $\mathcal{M}(A)$. 
The\textit{ Fourier transform} of a measure $\mu$ is defined by
\begin{align*}
    \widehat{\mu}(\xi)\coloneqq\int e^{-2\pi i \xi\cdot x}\,d\mu(x), \quad \xi\in \mathbb{R}^d.
\end{align*}
We write $X\lesssim Y$ if there exists a constant $C>0$ such that $X\leq CY$. If $X\lesssim Y$ and $Y\lesssim X$, we write $X\approx Y$.

Let $\Omega$ be a manifold (or $\R^d$), the standard spaces of $C^\infty$ functions on $\Omega$ and those of compact support are denoted by $ \mathcal{E}(\Omega)$ and $\mathcal{D}(\Omega)$, respectively, and $\mathcal{E}'$, $\mathcal{D}'$ stand for their dual spaces of distributions.  For $p\geq 1$, denote $C^\infty(\Omega,\R^p)$ as the space of all smooth maps $f: \Omega\to \R^p$; and $C^\omega(\Omega,\R^p)$ stands for the space of all analytic maps.

For $\sigma\in \R$, define the \textit{$L^2-$based Sobolev space} of 
order $\sigma$ on $\R^d$ to be the space of all tempered distributions $f$ such that the distribution $\widehat{f}(\xi)(1+\vert\xi\vert^2)^{\sigma/2}$ 
lies in $L^2$, and give this space the norm
\begin{align*}
    \left\Vert f\right\Vert_{L^2_\sigma(\R^d)}:= \left\Vert \widehat{f}(\xi)(1+\vert\xi\vert^2)^{\sigma/2}\right\Vert_{L^2(\R^d)}.
\end{align*}
Let $\Omega$ be a compact manifold of dimension $\dim\Omega=d$. If $f\in \mathcal{D}'(\Omega)$, we say $f\in L^2_{\sigma}(\Omega)$ provided that, 
for any coordinate patch $U\subset \Omega$, any $\psi\in C^\infty(U)$, the element $\psi f\in \mathcal{E}'(U)$ belongs to $L^2_\sigma(U)$, 
if $U$ is identified with its image on $\R^d$.

\section{Background on Microlocal Analysis and Geometric Measure Theory}\label{sec: preliminaries}
We first give a short survey of the FIO theory needed, see H{\"o}rmander \cite{Hormander1971,Hormanderbook}; 
then, we will recall basic background on dimensions of sets. In the next section, we will apply these tools to study Falconer-type and Mattila-Sj\H{o}lin type problems for 
a general class of configuration maps that satisfy certain nondegenerate conditions.

\subsection{Fourier integral operators}
Let $X$ be a $C^\infty$ manifold of dimension $d$. The cotangent bundle $T^\ast X$ is a \textit{symplectic manifold} with respect to the \textit{canonical symplectic two-form}, $\omega=\sum dx_j\wedge d\xi_j$ (with respect to any local coordinates). The zero section  $\left\{ \xi=0\right\}$ of $T^\ast X$ will be denoted by $\bzero$. A submanifold $\Lambda\subset T^\ast X\setminus \bzero$ is called \textit{conic} if it is invariant under $(x,\xi)\mapsto (x,\tau\xi)$ for $0<\tau<\infty$. If $\Lambda$ is smooth, conic, $\dim (\Lambda)=d=\frac{1}{2}\dim(T^\ast X)$, and $\omega\vert_\Lambda\equiv 0$, then we call $\Lambda $ a \textit{conic Lagrangian submanifold} of $T^\ast X$.

Assume that $X$ and $Y$ are smooth manifolds of dimensions $d_X$, $d_Y$, respectively. Then $T^\ast X$, $T^\ast Y$ are symplectic manifolds, with canonical symplectic two-forms $\omega_{T^\ast X},\omega_{T^\ast Y}$, respectively. Then $\omega_{T^\ast X}-\omega_{T^\ast Y}=\sum dx_j\wedge d\xi_j-\sum dy_j\wedge d\eta_j$ is the difference symplectic form on $T^\ast X\times T^\ast Y$. A \textit{canonical relation} is a submanifold $\Lambda\subset (T^\ast X\setminus \bzero)\times (T^\ast Y\setminus 0)$ which is conic Lagrangian with respect to $\omega_{T^\ast X}-\omega_{T^\ast Y}$.

Let $p\geq 1$ be an integer and let $\phi: X\times Y\times (\R^p\setminus \bzero)\to \R$ be a smooth phase function that is positively homogeneous of degree $1$ in $\theta\in \R^p$, i.e., $\phi(x,y,\tau\theta)=\tau\cdot \phi(x,y,\theta)$, for all $\tau\in \R_{+}$. Let $\Sigma_\phi$ be the \textit{critical set} of $\phi$ in the $\theta$ variables,
\begin{align*}
    \Sigma_\phi:=\left\{  (x,y,\theta)\in X\times Y\times (\R^p\setminus \bzero): d_\theta\phi(x,y,\theta)=0\right\}.
\end{align*}
The phase function $\phi(x,y,\theta)$ is said to be \textit{nondegenerate} if
\begin{align*}
    d_x\phi(x,y,\theta)\neq 0\quad \text{ and }\quad d_y\phi(x,y,\theta)\neq 0,\quad \forall (x,y,\theta)\in \Sigma_\phi,
\end{align*}
and 
\begin{align*}
    \rank\left[d_{x,y,\theta}d_\theta\phi(x,y,\theta)\right]=p,\quad \forall(x,y,\theta)\in \Sigma_\phi.
\end{align*}
This means that the $p$ differentials $\left\{ d(d_{\theta_l}\phi)\right\}_{l=1}^p$ are linearly independent, which implies that $\Sigma_\phi$ is a $C^\infty$ conic submanifold of $X\times Y\times (\R^p\setminus 0)$ with codimension $p$. Moreover, one can check that the map
\begin{align*}
    F:\Sigma_\phi\ni (x,y,\theta) \to (x,d_x\phi(x,y,\theta); y,-d_y\phi(x,y,\theta))\in (T^\ast X\setminus\bzero)\times (T^\ast Y\setminus \bzero)
\end{align*}
is an immersion, whose image $C_\phi:=F(\Sigma_\phi)$ is an immersed canonical relation; the phase function $\phi$ is said to \textit{parametrize} $C_\phi$.

Let $\Lambda\subset (T^\ast X\setminus \bzero)\times (T^\ast Y\setminus \bzero)$ be a canonical relation and $m\in \R$, we define $I^m(X,Y;\Lambda)$, the class of \textit{Fourier integral operators} $A:\mathcal{E}'(Y)\to \mathcal{D}'(X)$ of order $m$, as the collection of operators whose Schwartz kernels are locally finite sums of oscillatory integrals of the form
\begin{align*}
    K_A(x,y)=\int_{\R^p}e^{i \phi(x,y,\theta)}a(x,y,\theta)d\theta,
\end{align*}
where $\phi$ is a nondegenerate phase function as above paramatrizing some relatively open $C_\phi\subset \Lambda$, and the amplitude $a(x,y,\theta)$ is a symbol of order $\gamma=m-\frac{p}{2}+\frac{d_X+d_Y}{4}$.

For a general canonical relation $\Lambda\subset (T^\ast X\setminus \bzero)\times (T^\ast Y\setminus \bzero)$, consider the natural projections $\pi_L: T^\ast X\times T^\ast Y\to T^\ast X$, and $\pi_R: T^\ast X\times T^\ast Y\to T^\ast Y$ restricted to $\Lambda$, by abuse of notation we refer to restricted maps with the same notation. One can check that, at any point $c_0\in \Lambda$, one has $\corank(D\pi_L)(c_0)=\corank(D\pi_R)(c_0)$. The canonical relation $\Lambda$ is said to be \textit{nondegenerate} if this $\corank(D\pi_L)=0$ at all points of $\Lambda$, that is, if the differentials $D\pi_L$ and $D\pi_R$ have maximal rank. If $\dim X=\dim Y$, then $\Lambda$ is nondegenerate if and only if $\pi_L,\pi_R$ are local diffeomorphisms, and then $\Lambda$ is a local \textit{canonical graph}, i.e., locally near any $c_0\in \Lambda$, one has $\Lambda$ equal to the graph of a canonical transformation. If $\dim X> \dim Y$, then the condition that $\Lambda$ is nondegenerate is equivalent to $\pi_L$ is an immersion and $\pi_R$ is a submersion. 

We will write the order of the FIO $A\in I^m(X,Y;\Lambda)$ as $m=m_{\text{eff}}-\frac{\vert d_X-d_Y\vert}{4},$
where the effective order of $\mathcal{R}_\bt^\sigma$ is defined as
\begin{align*}
    m_{\text{eff}} := m+\frac{\vert d_X-d_Y\vert}{4}.
\end{align*}
By standard $L^2$-Sobolev estimates for FIOs, one has the following.
\begin{theorem}[H{\"o}rmander \cite{Hormander1971, Hormanderbook}]\label{thm:_FIO_standard}
    Suppose that $\Lambda\subset (T^\ast X\setminus \bzero)\times (T^\ast Y\setminus \bzero)$ is a canonical relation, where $\dim(X)=d_X$ and $\dim(Y)=d_Y$, and $A\in I^{m_{\text{eff}}-\frac{\vert d_X-d_Y\vert}{4}}(X,Y;\Lambda)$ has a compactly supported Schwartz kernel.
    \begin{itemize}
        \item[(i)] If $\Lambda$ is nondegenerate, then $A: L^2_s(Y)\to L^2_{s-m_{\text{eff}}}(X)$ for all $s\in \R$. Furthermore, the operator norm depends boundedly on a finite number of derivatives of the amplitude and phase function.
        \item[(ii)] If the spatial projections from $\Lambda$ to $X$ and to $Y$ are submersions and, for some $q$, the corank of $D\pi_L$ (and thus that of $D\pi_R$) is $\leq q$ at all points of $\Lambda$, then $A: L^2_s(Y)\to L^2_{s-m_{\text{eff}}-\frac{q}{2}}(X)$ for all $s\in \R$. 
    \end{itemize}
\end{theorem}
For certain classes of degenerate canonical relations, if the projections $\pi_L$ and/or $\pi_R$ degenerate in specific ways, then the loss of derivatives is often less than $q/2$. The first and best known result is the following, which is due to Melrose and Taylor \cite{Melrose_Taylor_1985}. Suppose $\dim X=\dim Y=d$, then a canonical relation $\Lambda\subset (T^\ast X\setminus \bzero)\times (T^\ast Y\setminus \bzero)$ is said to be a \textit{folding canonical relation} (also called \textit{two-sided folds} \cite{Greenleaf_Seeger_1994}) if at any degenerate points $c_0\in \Lambda$, both $\pi_L$ and $\pi_R$ have Whitney fold singularities.
\begin{theorem}[Melrose and Taylor \cite{Melrose_Taylor_1985}]\label{thm: FIO_Whitney_singularities}
    If $\Lambda$ is a folding canonical relation and $A\in I^m(X,Y;\Lambda)$ with compactly supported Schwartz kernel, then $A: L^2_s(Y)\to L^2_{s-m-(1/6)}(X)$ for all $s\in \R$.
\end{theorem}

\subsection{Dimensions of sets}
The following lemma gives an equivalent definition for the Hausdorff dimension.
\begin{lemma}[Frostman's lemma, Theorem 2.7, \cite{Mattila2015}]
Let $0\leq s\leq d$.
For a Borel set $A\subset \mathbb{R}^d$, the $s-$dimensional Hausdorff measure of $A$ is positive if and only if there exists a measure $\mu\in\mathcal{M}(A)$ satisfying
\begin{align}\label{eq: frostman}
    \mu(B(x,r))\lesssim r^s, \quad \forall\, x\in \mathbb{R}^d,\,\,r>0.
\end{align}
In particular,
\begin{align*}
    \dim_{H} A
    &\coloneqq \sup\left\{  s\leq d: \exists \mu\in \mathcal{M}(A) \text{ such that } \eqref{eq: frostman} \text{ holds}\right\}.
\end{align*}
\end{lemma}
Frostman's lemma yields that for any exponent $0<s< \dim_H (A)$, there exists a probability measure $\mu$ on $A$ satisfying \eqref{eq: frostman}. 
The \textit{$s$-energy integral} of a measure $\mu\in \mathcal{M}(\R^d)$ (see \cite{Mattila2015})
is
\begin{align*}
    I_s(\mu)\coloneqq\iint\vert x-y\vert^{-s}\,d\mu(x)d\mu(y) = c(n,s)\int\vert\widehat{\mu}(\xi)\vert^2\vert\xi\vert^{s-n}d\xi.
\end{align*}
If $\mu$ is a Frostman measure satisfying $\eqref{eq: frostman}$, then $I_t(\mu)<\infty$ for all $t\in (0,s)$. 
Since $\mu$ is of compact support, $\widehat{\mu}\in C^\omega$, and thus this implies
\begin{align*}
    \mu\in L_{(t-d)/2}^2(\R^d).
\end{align*}
This also holds in the general setting when $A\subset X$ is a compact subset of a $d-$dimensional manifold $X$ with $\dim_H(A)>t$.

We recall the following elementary result from \cite{GreenleafIosevichTaylor22} on the tensor products of Sobolev spaces.

\begin{proposition}
    For $1\leq j\leq k$, let $X^j$ be a $C^\infty$ manifold of dimension $d_j$, and suppose that \linebreak $u_j\in L_{s_j,comp}^2(X^j)$, $1\leq j\leq k$, with each $s_j\in \R$. Then the tensor product $u_1\otimes \cdots \otimes u_k$ belongs to $L_{s,comp}^2(X^1\times \cdots X^k)$, for $s=\sum\limits_{j=1}^k s_j$.
\end{proposition}
The \textit{Sobolev energy} of degree $s\in \R$ of $\mu\in \mathcal{M}(\R^d)$
is \[\calI_s(\mu):= \int \vert\widehat{\mu}(x)\vert^2(1+\vert x\vert)^{s-d}dx.\]
When $0<s<d$, $I_s(\mu)=\calI_s(\mu)$.
The \textit{Sobolev dimension} of a measure $\mu\in \mathcal{M}(\R^d)$ is
\begin{align*}
    \dim_S\mu:=\sup\left\{  s\in \R:  \int \vert\widehat{\mu}(x)\vert^2(1+\vert x\vert)^{s-d}dx<\infty\right\}.
\end{align*}
We then say that a set $A\subset \R^d$ has Sobolev dimension $\dim_SA\geq s$ if $A$ carries a Borel probability measure $\mu$ such that $\dim_S\mu\geq s$. The greater the Sobolev dimension is, the smoother the measure is in some sense. The following result is well-known, see \cite{Mattila2015}.
\begin{proposition}[{\cite[Theorem 5.4]{Mattila2015}}]\label{prop:sobolev_dim}
    Let $\mu\in \mathcal{M}(\R^d)$.
    \begin{itemize}
        \item[(i)] If $\, 0<\dim_S\mu<d$, then $\dim_S\mu=\sup\left\{  s>0: I_s(\mu)<\infty\right\}$.
        \item[(ii)] If $\dim_S\mu>d$, then $\mu\in L^2(\R^d)$.
        \item[(iii)] If $\dim_S\mu>2d$, then $\mu$ is a continuous function. 
    \end{itemize}
\end{proposition}
Let us also recall the following results from \cite{Raz_Zahl_2024}.  Let $A\subset \R^d$, $d\geq 1$, let $x\in \R^d$, and $\alpha>0$. We say that $A$ has \textit{local dimension $\geq \alpha$} at $x$ if $\dim(A\cap U)\geq \alpha$ for every neighborhood $U$ of $x$. Otherwise, we say $A$ has local dimension $<\alpha$ at $x$.
\begin{lemma}[{\cite[Lemma 9.2]{Raz_Zahl_2024}} ]\label{lem: local_dim}
    Let $A\subset \R^d, d\geq 1$ be such that $\dim_HA>0$. Then for each $\e>0$, there exists at least one point $x\in A$ such that $A$ has local dimension $\geq \dim_HA-\e$.
\end{lemma}
\begin{lemma}[{\cite[Lemma 9.4]{Raz_Zahl_2024}} ]\label{lem: local_dim_away_analytic_variety}
    Let $d\geq 1$. Let $\Omega\subset \R^d$ be a connected open set, let $g\in C^\omega(\Omega,\R)$ be analytic and not identically zero on $\Omega$, and let $\e>0$.  Let $A=A_1\times A_2\times\dots\times A_d\subset \Omega$ be a compact set, where each $A_i\subset \R$ has positive Hausdorff dimension. Then there is
    a point $a_0=(a_1,\dots, a_d)\in \Omega\setminus g^{-1}(0)$ so that for each index $i$, $A_i$ has local dimension $\geq \dim_H A_i-\e$ at $a_i$, i.e.
  \begin{align}\label{eq: base_point}
      \dim_H (A_i\cap  U)\geq\dim_HA_i-\e,\quad \text{ for every neighborhood $U$ of $a_i$}.
  \end{align}
\end{lemma}

\section{\texorpdfstring{$k$}{k}-point \texorpdfstring{$\Phi$}{\Phi}-configuration sets}
\addtocontents{toc}{\protect\setcounter{tocdepth}{2}}
\label{sec: k_point_configurations}
We begin this section by recalling the FIO techniques developed in \cite{GreenleafIosevichTaylor2021,GreenleafIosevichTaylor22,GreenleafIosevichTaylor2024}
and then modify them to study the Hausdorff dimensions or positive Lebesgue measure of configuration sets
in Theorem \ref{thm: general_phi_configurations} (i), (ii), respectively. 
Next, some explicit nondegeneracy conditions for the map $\Phi$ in the case of codimension $p=1$ will be discussed. 

\subsection{General configuration functions}
Suppose that $k\geq 2$, and $M_i$, $1\leq i\leq k$, are smooth manifolds of dimensions $\dim M_i=n_i$. Denote $M=M_1\times \dots \times M_k$, then $M$ is a smooth manifold of dimension $\dim M=\sum\limits_{i=1}^kn_i=n$. Let $\Phi\in C^\infty(M,\R^p)$ be a smooth map, with $p\geq 1$. We further assume that $\Phi$ is a submersion, that is, the full differential $D\Phi$ has maximal rank equal to $p$ everywhere. 
 
 Let $A_i\subset M_i$, $1\leq i\leq k$, be compact sets, and set $A=A_1\times\dots \times A_k$. Then denote\[\Delta_\Phi(A)=\left\{ \Phi(a^1,\dots, a^k): a^i\in A_i, 1\leq i\leq k\right\}\] as the $k-$point $\Phi$-configuration set of the $A_i$. 
 
 Our goal is to find suitable nondegeneracy conditions for $\Phi$ and a lower bound for $\sum\limits_{i=1}^k\dim_H A_i$ ensuring that $\calL^p\left(\Delta_\Phi(A)\right)>0$.
 More generally, fixing $0<u\leq p$, we will find the lower bound for $\sum\limits_{i=1}^k\dim_H A_i$ such that $\dim_H\Delta_\Phi(A)\geq u$.  

To this end, first, put $X=M\times M$ and refer to it as $X=X_1\times\dots \times X_{2k}$, where $X_i=X_{k+i}=M_i$, for $1\leq i\leq k$. Denote the dimension of each $X_i$ as $d_i$, for $1\leq i\leq 2k$, and the dimension of $X$ is $d_{tot}=\sum\limits_{i=1}^{2k}d_i=2n$. For each $x\in X$, write $x=\left(x',x''\right)=\left(x^1,\dots,x^{2k}\right)$. Define a smooth map $\Psi=\Psi_\Phi\in C^\infty(X,\R^p)$ by 
\begin{align*}
    \Psi(x',x''):=\Phi(x')-\Phi(x''), \quad \forall (x',x'')\in X.
\end{align*}
    Let $\phi:X\times (\R^p\setminus\bzero)\to \R$  be a smooth function defined by $\phi(x',x'',\theta)=\theta\cdot\big(\Phi(x')-\Phi(x'')\big)$ for $(x',x'',\theta)\in X\times (\R^p\setminus \bzero)$. Then the critical set of $\phi$ in the $\theta$ variables is 
    \begin{align*}
        \Sigma_\phi&:=\left\{ (x',x'',\theta)\in X\times (\R^p\setminus \bzero): d_\theta\phi(x',x'',\theta)=0\right\}\\
        &=\left\{ (x',x'',\theta)\in X\times (\R^p\setminus \bzero): \Phi(x')=\Phi(x'')\right\}.
    \end{align*}
    Since $\Phi$ is a submersion, it is easy to see that the phase function $\phi(x',x'',\theta)$ is nondegenerate, that is,
    \begin{align*}
        &d_{x'}\phi\neq0,\,d_{x''}\phi\neq 0,\quad \forall (x',x'',\theta)\in \Sigma_\phi,
    \end{align*}
    and
    \begin{align*}
        &\rank \left[d_{x',x'',\theta}\,d_\theta\phi(x',x'',\theta) \right]=p,\quad \forall (x',x'',\theta)\in \Sigma_\phi.
    \end{align*}
    Thus, the level set $Z:=\Psi^{-1}(0)=\left\{ (x',x'')\in X: \Phi(x')=\Phi(x'')\right\}$ is a smooth submanifold of $X$ with codimension $p$. Observe that $\phi$ parametrizes the canonical relation 
    \begin{align*}
        \Lambda:=N^\ast Z=\left\{  \big(x',x'', D\phi(x',x'')^\ast (\theta)\big): (x',x'')\in Z,\theta\in \R^p\setminus \bzero\right\}.
    \end{align*}
    Let $a(x',x'',\theta)\in C^\infty(X\times \R^p)$ be a symbol of order $\gamma=u-p$ defined by 
    \begin{align}\label{eq: symbol_a}
        a(x',x'',\theta):=\chi(x',x'')(1+\vert\theta\vert)^{u-p},\quad \forall (x',x'',\theta)\in X\times \R^p,
    \end{align}
where $\chi: M\times M\to \R$ is a compactly supported smooth function such that $\chi\equiv 1$ on $A\times A$. Define
    \begin{align}\label{eq: Schwartz_kernel_Ku}
        K(x',x''):=\int e^{-2\pi i \phi(x',x'',\theta)}a(x',x'',\theta)\,d\theta, \quad (x',x'')\in X.
    \end{align}
    Then $K$ is a Fourier integral distribution on $X$. In other words, in H\"ormander's notation,
    \begin{align}\label{eq: FIO_order_K}
        K\in I^{u-p/2-n/2}(X;N^\ast Z),
    \end{align}
    where the value of the order follows from the amplitude having order $\gamma=u-p$, and the numbers of phase variables and spatial variables being $p$ and $d_{tot}=2n$, respectively, so that the order is 
    \begin{align*}
        m=\gamma+p/2-d_{tot}/4= u-p+p/2-2n/4=u-p/2-n/2.
    \end{align*}
    Denote $\mathcal{P}_{2k}$ as the set of all nontrivial bipartite partitions of $\left\{ 1,\dots,2k\right\}$, and write such a partition as $\sigma=(\sigma_L\mid\sigma_R)$ with $\vert\sigma_L\vert,\vert\sigma_R\vert>0$, and $\vert\sigma_L\vert+\vert\sigma_R\vert=2k$. For each $\sigma\in \mathcal{P}_{2k}$, we write $\sigma_L=\left\{ i_1,\dots, i_{\vert\sigma_L\vert}\right\}$, and $\sigma_R=\left\{ j_1,\dots, j_{\vert\sigma_R\vert}\right\}$, assuming that $i_1<\dots<i_{\vert\sigma_L\vert}$ and $j_1<\dots<j_{\vert\sigma_R\vert}$. 
Now we separate the variables $x^1,\dots, x^{2k}$ into groups on the left and right, and use the same notation $x$ for the coordinate-partitioned version of $x$,
\begin{align*}
    x=(x_L;x_R)=\left(x^{i_1},\dots, x^{i_{\vert\sigma_L\vert}};x^{j_1},\dots, x^{j_{\vert\sigma_R\vert}}\right).
\end{align*}
The corresponding reordered Cartesian product of the $X^i$ is
\begin{align*}
    X_L\times X_R=\big(X_{i_1}\times \dots\times X_{i_{\vert\sigma_L\vert}}\big)\times \big(X_{j_1}\times \dots\times X_{j_{\vert\sigma_R\vert}}\big),
\end{align*}
and we still refer to this as $X$. Note that the dimensions of the two manifolds are $\dim X_L=d_{\sigma,L}=\sum\limits_{i\in \sigma_L} d_i$, and $\dim X_R=d_{\sigma,R}=\sum\limits_{j\in \sigma_R} d_j$.
The coordinate-partitioned version of $Z$ associated with $\sigma$ is denoted by $Z^\sigma$, i.e.,
\begin{align*}
    Z^\sigma =\left\{ (x_L,x_R): \Psi(x)=0\right\}\subset X_L\times R_R.
\end{align*}
Let $\pi_{L}:Z^\sigma\to X_L$, and $\pi_{R}:Z^\sigma\to X_R$ be the spatial projections to the left and right. 
We say that $Z^\sigma$ satisfies the \textit{double fibration condition} $(DF)_\sigma$ if the two spatial projections from $Z^\sigma$ have maximal rank (see \cite{Guillemin77,Guillemin1990,Helgason1965}), namely,
\begin{align}\label{eq: doublefibration}
 \pi_L: Z^\sigma\to X_L,\quad \text{and }\quad \pi_R: Z^\sigma\to X_R \quad \text{ are submersions}.
\end{align}
\begin{sloppypar} 
It is easy to see that if 
    \[\rank (D_{x_L}\Psi)=p \quad \text{ and } \quad \rank(D_{x_R}\Psi)=p\]
everywhere, then $(DF)_\sigma$ is satisfied, and one must have the necessary condition that $p\leq \min\left\{ d_{\sigma,L}, d_{\sigma,R}\right\}$.
\end{sloppypar}
The Schwartz kernel $K$ can now be written as
\begin{align*}
    K^\sigma(x_L,x_R)=\int e^{-2\pi i \phi_\sigma(x_L,x_R,\theta)}a_\sigma(x_L,x_R,\theta)\, d\theta,
\end{align*}
where $\phi_\sigma(x_L,x_R,\theta)=\phi(x',x'',\theta)$, and $a_\sigma(x_L,x_R,\theta)=a(x',x'',\theta)$, for all $(x_L,x_R)\in X_L\times X_R$, and $\theta\in \R^p$. We also assume that $\phi_\sigma$ is nondegenerate, that is,
\begin{align*}
    d_{x_L}\phi_\sigma(x_L,x_R,\theta)\neq 0 \quad \text{ and  }\quad d_{x_R}\phi_\sigma(x_L,x_R,\theta)\neq 0
\end{align*}
on the set $\Sigma_{\phi_\sigma}=\left\{  (x_L,x_R,\theta)\in X_L\times X_R\times \R^p\setminus \bzero: \Psi(x_L,x_R)=0\right\}$.

Next, the operator $\calT^\sigma$ is defined weakly by
\begin{align}\label{eq: op_T_sigma}
    \calT^\sigma f(x_L):=\int K^\sigma(x_L,x_R)f(x_R)\,dx_R.
\end{align}
The operator $\calT^\sigma$ extends from the mapping $\mathcal{D}(X_R)\to \mathcal{E}(X_L)$ to 
\begin{align*}
    \calT^\sigma: \mathcal{E}'(X_R)\to \mathcal{D}'(X_L).
\end{align*}
Moreover, $\calT^\sigma$ is an FIO associated with a canonical relation $\Lambda^\sigma$, where $\Lambda^\sigma\subset (T^\ast X_L\setminus\bzero)\times (T^\ast X_R\setminus \bzero)$ is the twisted conormal bundle of $Z^\sigma$, that is,
\begin{align*}
    \Lambda^\sigma
    &:=(N^\ast Z^\sigma)'=\left\{  (x_L,\xi_L;,x_R,\xi_R):(x_L,x_R)\in Z^\sigma, (\xi_L,-\xi_R)\perp TZ^\sigma, \xi_L,\xi_R\neq \bzero\right\}.
\end{align*}
In H\"ormander's notation, $\calT^\sigma\in I^{u-p/2-n/2}(X_L,X_R;\Lambda^\sigma)$, where the order $m=u-p/2-n/2$ is determined as in \eqref{eq: FIO_order_K}. Then the effective order of $\calT^\sigma$ is
\begin{align*}
    m_{\text{eff}}^\sigma = m+\frac{\vert d_{\sigma,L}-d_{\sigma,R}\vert}{4}=\frac{-\min\left\{ d_{\sigma,L},d_{\sigma,R}\right\}-p+2u}{2}.
\end{align*}

Let $\pi_L:\Lambda^\sigma\to T^\ast X_L$, and $\pi_R:\Lambda^\sigma\to T^\ast X_R$ be natural projections of $\Lambda^\sigma$ to the cotangent bundles $T^\ast X_L$ and $T^\ast X_R$, respectively.
If $\Lambda^\sigma$ is a nondegenerate canonical relation, that is, $\corank (D\pi_L)=0$ everywhere, it follows from Theorem \ref{thm:_FIO_standard} that
\begin{align*}
    \calT^\sigma: L_s^2(X_R)\to L_{s-m_{\text{eff}}^\sigma}^2(X_L),\quad \text{ for all } s\in \R.
\end{align*}
More generally, if $\corank(D\pi_L)\leq q$ at all points of $\Lambda^\sigma$, then there is a loss of at most $q/2$ derivatives, i.e.,
\begin{align*}
    \calT^\sigma: L_s^2(X_R)\to L_{s-m_{\text{eff}}^\sigma-\frac{q}{2}}^2(X_L),\quad \text{ for all } s\in \R.
\end{align*}
In general, the operator $\calT^\sigma$ is said to satisfy the condition $[(DF)_\sigma;\beta^\sigma]$, for $\beta^\sigma\geq 0$, if the double fibration condition $(DF)_\sigma$ holds, and for all $s\in \R$,
\begin{align}\label{eq: condition_FIO_beta_sigma}
\calT^\sigma: L_s^2(X_R)\to L_{s-m_{\text{eff}}^\sigma-\beta^\sigma}^2(X_L).
\end{align}
We say that $\mathcal{U}$ is a \textit{conic neighborhood} of $(x,\xi)$ in $N^\ast Z$ if for all $(y,\eta)\in \mathcal{U}$, we have $(y,\tau\eta)\in \mathcal{U}$, $\forall \tau>0$.

Now we state our first main result in this section. Recall that $M=M_1\times \dots \times M_k$ is a manifold with $\dim M=n$.

\begin{theorem}\label{thm: general_phi_configurations}
Let $k\geq 2$, $p\geq 1$, $\alpha,\beta\geq 0$, and $0<u<\frac{1}{2}(2n+p-\alpha-2\beta)$. Let $\Phi\in C^\infty(M,\R^p)$ be a smooth map. Assume that for every point $(x,\xi)\in N^\ast Z$, there exist a partition $\sigma\in \mathcal{P}_{2k}$ such that $\max (d_{\sigma,L},d_{\sigma,R})\leq \alpha$, and a conic neighborhood $\,\mathcal{U}$ of $(x,\xi)$ for which the operator $\calT^\sigma$, microlocalized to $\mathcal{U}$, satisfies $[(DF)_\sigma;\beta]$ with a loss of at most $\beta$ derivatives.

Suppose that $A_i\subset M_i$, $1\leq i\leq k$, are compact sets such that each $\dim_HA_i>0$.
\begin{itemize}
\item[(i)] For any $\, 0<u<p$, if
    $\sum\limits_{i=1}^k\dim_H A_i>\frac{\alpha-p}{2}+u+\beta$,
then $\dim_H\left(\Delta_\Phi(A)\right)\geq u$.
    \item[(ii)] If $u\geq p$, and
    $\sum\limits_{i=1}^k\dim_H A_i>\frac{\alpha+p+2\beta}{2}$,
then $\calL^p\left(\Delta_\Phi(A)\right)>0$.
\item[(iii)] If $u> 2p$, and 
    $\sum\limits_{i=1}^k\dim_H A_i>\frac{\alpha-p}{2}+2p+\beta$,
then $\Int\left(\Delta_\Phi(A)\right)\neq \emptyset$.
\end{itemize}
\end{theorem}
\begin{remark}
    In the above theorem, note that $\alpha\geq \max (d_{\sigma,L},d_{\sigma,R})\geq n$. In applications, we optimize the choice of $\sigma=(\sigma_L\vert \sigma_R)\in \mathcal{P}_{2k}$ such that $\calT^\sigma$ satisfies the hypotheses of the theorem with $\alpha$ minimizing $\max\left\{ d_{\sigma,L},d_{\sigma,R}\right\}$, and $\beta$ being close to $0$.
\end{remark}
\begin{proof}
Let $p\geq 1$, $u> 0$, and let $\Phi: M\to \R^p$ be a smooth map satisfying the hypotheses of the theorem with constants $\alpha,\beta\geq 0$, such that $\alpha+2\beta+2u<2n+p$. Assume that $A_i\subset M_i$, $1\leq i\leq k$, are compact sets such that
\begin{align*}
    \sum\limits_{i=1}^k\dim_HA_i>\frac{\alpha-p}{2}+u+\beta.
\end{align*}
Thus, there exist $0<s_i<\dim_HA_i$, $1\leq i \leq k$, such that 
\begin{align}\label{eq: bound_dimensions_s_i}
    \sum\limits_{i=1}^k s_i> \frac{\alpha-p}{2}+u+\beta.
\end{align}

By Frostman's lemma, for each $i$, there exists $\mu_i\in \mathcal{M}(A_i)$ such that $I_{s_i}(\mu_i)<\infty$.
    Define a natural measure $\nu$ supported on $\Delta_\Phi(A)$ by the relation
    \begin{align*}
        \int_{\R^p} f(\bt)\,d\nu(\bt)= \int f\big(\Phi(x^1,\dots, x^k)\big)\,d\mu_1(x^1)\dots d\mu_k(x^k),\quad \text{ for } f\in C_0(\R^p).
    \end{align*}
    In other words, $\nu$ is the pushforward measure of $\mu=\mu_1\times \dots\times \mu_k$ under the map $x\mapsto \Phi(x)$.

   In order to show that $\Delta_\Phi(A)$ has positive Lebesgue measure when $u\geq p$, it suffices to show that under the condition \eqref{eq: bound_dimensions_s_i}, one has
    \begin{align*}
        \int \vert\nu(\bt)\vert^2\,d\bt <\infty.
    \end{align*}
    This implies that $\nu\in L^2(\R^p)$, has a density also denoted as $\nu$, which is absolutely continuous with respect to Lebesgue measure. Since $\nu$ is supported on $\Delta_\Phi(A)$, this yields $\calL^p\big(\Delta_\Phi(A)\big)>0$ as desired.
    
    More generally, we will show that  under the condition \eqref{eq: bound_dimensions_s_i}, we have 
    \begin{align}\label{eq:L^2_Sobolev_dim}
        \int \vert\widehat\nu(\theta)\vert^2(1+\vert\theta\vert)^{u-p}\,d\theta <\infty.
    \end{align}
    Invoking Proposition \ref{prop:sobolev_dim}, we obtain the desired result.
   
    In order to prove \eqref{eq:L^2_Sobolev_dim}, observe that for $\theta\in \R^p$, by the definition of $\nu$, one has
    \begin{align*}
        \widehat\nu(\theta)=\int e^{-2\pi i\theta\cdot\bt}\,d\nu(\bt)=\int e^{-2\pi i \theta \cdot \Phi(x)}\,d\mu(x)
    \end{align*}
\begin{sloppypar}
Therefore, by Fubini's theorem, one has
    \begin{align*}
        \int \vert\widehat\nu(\theta)\vert^2(1+\vert\theta\vert)^{u-p}\,d\theta&= \int \bigg(\iint e^{-2\pi i \theta \cdot (\Phi(x')-\Phi(x''))}\,d\mu(x')\,d\mu(x'')\bigg)(1+\vert\theta\vert)^{u-p}\,d\theta\\ 
        &= \iint \bigg(\int e^{-2\pi i \phi(x',x'',\theta)}a(x',x'',\theta)\,d\theta\bigg)\,d\mu(x')\,d\mu(x'')\\
        &=\iint K(x',x'')\,d\mu(x') \,d\mu(x''),
    \end{align*}
    where the amplitude $a(x',x'',\theta)=\chi(x',x'')(1+\vert\theta\vert)^{u-p}$ is a symbol of order $\gamma=u-p$, (see \eqref{eq: symbol_a}), and the Schwartz kernel $K$ was defined in \eqref{eq: Schwartz_kernel_Ku}.         
\end{sloppypar}    

According to the assumptions of the theorem, for each point $(x,\xi)\in N^\ast Z$, there exist a partition $\sigma_l\in \mathcal{P}_{2k}$ such that $\max\left\{  d_{\sigma,L},d_{\sigma,R}\right\}\leq \alpha$, and a conic neighborhood $\mathcal{U}_l$ of $(x,\xi)$ for which the operator $\calT^\sigma$, microlocalized to $\mathcal{U}_l$, satisfies \eqref{eq: condition_FIO_beta_sigma} with the loss of at most $\beta$ derivatives. The collection $\left\{ U_l\right\}_l$ then forms an open cover for $N^\ast Z$. Let $\mathcal{U}_0$ be an open set such that $\mathcal{U}_0$ is disjoint from $N^\ast Z$ which completes $\left\{ \mathcal{U}_l\right\}$ to be a cover of $T^\ast X\setminus \bzero$. Let $\left\{ Q_l(x,D)\right\}_l$ be a standard pseudodifferential partition of unity on $X$ subordinate to the open cover $\left\{ \mathcal{U}\right\}_l$. For each $l$, $Q_l$ belongs to $ \Psi_{cl}^0(X)$, the set of all classical pseudo-differential operators on $X$ of order $0$, and their principal symbols, $q_l(x,\xi)$, form a partition of unity on $T^\ast X\setminus \bzero$. We can assume that the sum $\sum\limits_l Q_l(x,D)=I$ has a finite number of terms. For each $l$, recall that $\sigma_l$ is the partition such that $\calT^{\sigma_l}$ satisfies \eqref{eq: condition_FIO_beta_sigma} with the loss of $\leq \beta$ derivatives on the conic support of $Q_l$, and denote the set of all such $\sigma_l$ by $\mathcal{W}$.

Now for each $\sigma\in \mathcal{W}$, $\sigma=(\sigma_L\vert\sigma_R)$, define
\begin{align*}
    \mu_L^\sigma=\mu_{i_1}\times \dots\times \mu_{i_{\vert\sigma_L\vert}}\quad \text{ and}\quad \mu_R^\sigma=\mu_{j_1}\times \dots \times \mu_{j_{\vert\sigma_R\vert}}.
\end{align*}
By assumptions, one has $\mu_L^\sigma\in L_{s_L^\sigma}^2(X_L)$, and $\mu_R^\sigma\in L_{s_R^\sigma}^2(X_R)$, where $s_L^\sigma=\frac{1}{2}\sum\limits_{i\in \sigma_L} (s_i-d_i)$, and $s_R^\sigma=\frac{1}{2}\sum\limits_{j\in \sigma_R}(s_j-d_j)$, respectively. Since $\calT^\sigma$ satisfies \eqref{eq: condition_FIO_beta_sigma} with the loss $\leq \beta$ derivatives, one has
\begin{align*}
    \calT^\sigma(\mu_R^{\sigma})\in L_{s_R^{\sigma}-m_{\text{eff}}^{\sigma}-\beta}^2(X_L).
\end{align*}
Since $\mu_L^{\sigma}\in L_{s_L^{\sigma}}(X_L)$, if
\begin{align}\label{eq: exponent}
    s_R^{\sigma}-m_{\text{eff}}^{\sigma}-\beta+s_L^{\sigma}\geq 0,
\end{align}
then the pairing $\langle \calT^\sigma\mu_R^{\sigma}, \mu_L^{\sigma}\rangle$ is bounded. One can check that \eqref{eq: exponent} holds if and only if
\begin{align*}
    \sum\limits_{i=1}^{k}s_i\geq\frac{\max\left\{ d_{\sigma,L},d_{\sigma,R}\right\}-p}{2}+u+\beta,
\end{align*}
which follows from \eqref{eq: bound_dimensions_s_i}, as $\max\left\{ d_{\sigma,L},d_{\sigma,R}\right\}\leq \alpha$. Since this holds for all $\sigma\in \mathcal{W}$, one concludes that
\begin{align*}
    \int\vert\widehat\nu
(\theta)\vert^2(1+\vert\theta\vert)^{u-p}\,d\theta=\sum\limits_{\sigma\in \mathcal{W}}\langle \calT^\sigma\mu_R^{\sigma}, \mu_L^{\sigma}\rangle<\infty
\end{align*}
as desired. This completes the proof of the theorem.    
\end{proof}

\subsection{The case codimension \texorpdfstring{$p=1$}{p=1}}
Suppose that $k\geq 2$, and let  $U_i\subset \R^{n_i}$, $1\leq i\leq k$, be open subsets with $n_i\geq 1$, and again put $n=\sum\limits_{i=1}^k n_i$. Assume that
$\Phi: U_1\times \dots \times U_k\to \R$ is a smooth function. 
The goal of this section is to derive explicit nondegeneracy conditions for the function $\Phi$ when the codimension $p=1$.

\subsubsection{\texorpdfstring{$k$}{k}-point configurations with \texorpdfstring{$k\geq2$}{k\geq 2}}
We begin by introducing some definitions and notation. For $x\in\Omega= V_1\times \dots \times V_k$, write $x=(x^1,\dots,x^k)$, where
\begin{align*}
    x^i=\left(x^i_{1},\dots, x^i_{n_i}\right)\in V_i, \quad \text{ for } 1\leq i\leq k.
\end{align*}
Let $E=\left\{ i_1,\dots ,i_E\right\}$ and $F=\left\{ j_1,\dots, j_F\right\}$ be nonempty subsets of $\left\{ 1,\dots, k\right\}$ such that complements of $E$ and $F$ in $\left\{ 1,\dots,k\right\}$ are nonempty, denoted by $E'$ and $F'$, respectively. Denoting $x_E=\left(x^{i_1},\dots ,x^{i_E}\right)$, and similarly for $x_F$, $x_{E'}$, and $x_{F}'$. Letting $n_E=\sum\limits_{i\in E}n_i$, $n_F=\sum\limits_{j\in F}n_j$, and\[n(E,F)=\max\left\{ n_E+n_F, 2n-(n_E+n_F)\right\}.\]

We construct an $n_E\times n_F$ matrix with elements being mixed second-order partial derivatives of $\Phi$ with respect to variables $\left\{ x^i_{l_i}\right\}_{i\in E}$ and $\left\{  x^j_{l_j}\right\}_{j\in F}$, denoted by $H_{x_E,x_{F}}\Phi$ or $H_{E,F}\Phi$,
\begin{align*}  
H_{x_E,x_{F}}\Phi:=\bigg[\frac{\partial^2\Phi}{\partial x^i_{l_i}\partial x^j_{l_j}
    }\bigg]_{i\in E, j\in F}=\begin{bmatrix}
        \vdots&\vdots&&\vdots\\
        \frac{\partial^2\Phi}{\partial x^i_{l_i}\partial x^{j_1}_{1}}&\frac{\partial^2\Phi}{\partial x^i_{l_i}\partial x^{j_1}_{2}}&\cdots &\frac{\partial^2\Phi}{\partial x^i_{l_i}\partial x^{j_F}_{n_{j_F}}}
    \\
        \vdots&\vdots&&\vdots
    \end{bmatrix}_{i\in E}.
\end{align*}
This matrix is then called the \textit{mixed Hessian of $\Phi$} with respect to $\left\{ E,F\right\}$. The mixed Hessian matrices $ H_{x_E,x_{E'}}\Phi$ and $ H_{x_F,x_{F'}}\Phi$ are defined similarly.

Let $\Psi(x,y):=\Phi(x)-\Phi(y)$, for $(x,y)\in \Omega\times\Omega$. If $\Phi$ is a submersion, it follows that the level set $Z=\Psi^{-1}(0)=\left\{  (x,y)\in \Omega\times\Omega: \Phi(x)=\Phi(y)\right\}$ is a smooth submanifold of codimension $1$ in $\Omega\times\Omega$. We further assume that
\begin{align}\label{eq: gradient_condition}
    (\nabla_{x_{E'}}\Phi,\nabla_{y_{F'}}\Phi)\neq (\bzero,\bzero) \quad\text{ and  }\quad(\nabla_{x_E}\Phi,\nabla_{y_F}\Phi)\neq (\bzero,\bzero)
\end{align}
on $Z$. 

\begin{definition}\label{def: Gamma_m(E,F)_nondegenerate}
    Let $\left\{ E,F\right\}$ be nonempty subsets of $\left\{ 1,\dots,k\right\}$. For $(x,y)\in \Omega\times \Omega$, define an $(n_E+n_F+1)\times(2n-n_E-n_F+1)$ matrix
\begin{align}\label{eq: matrix_J}
    J_{E,F}\Phi:=\begin{bmatrix}
    0&\nabla_{x_{E'}}\Phi(x) &-\nabla_{y_{F'}}\Phi(y)\,\\
       (\nabla_{x_E}\Phi(x))^T &H_{x_E,x_{E'}}\Phi(x)&0\,\\
       (\nabla_{y_F}\Phi(y))^T &0 &H_{y_F,y_{F'}}\Phi(y)\,
    \end{bmatrix}.
\end{align}
For $(x_0,y_0)\in Z$, we say that $\Phi$ is \textit{$\Gamma_m(E,F)-$} \textit{nondegenerate} at $(x_0,y_0)$ if the gradient condition \eqref{eq: gradient_condition} holds and there exists a nonnegative integer $m$ such that $\corank(J_{E,F}\Phi)\leq m$ in a small neighborhood of $(x_0,y_0)$.
\end{definition}

The following theorem gives explicit conditions for the function $\Phi$ such that one can obtain Falconer-type results.
\begin{theorem}\label{thm: k_point_explicit_p=1}
    Let $\, \Phi\in C^\infty(\Omega,\R)$ be a smooth function.  Assume that there exist integers $\alpha\geq n$, $m\geq 0$ such that the following holds:
    for any $(x_0,y_0)\in Z$, there exist nonempty subsets $E,F$ of  $\left\{ 1,\dots, k\right\}$ such that $n_{E,F}\leq \alpha$, and $\Phi$ is $\Gamma_m(E,F)-$nondegenerate at $(x_0,y_0)$. 

    Suppose that $A_i\subset U_i$, $1\leq i\leq k$, are compact sets such that each $\dim_HA_i>0$.
\begin{itemize}
\item[(i)] For any $\,0<u<1$, if $\sum\limits_{i=1}^k
\dim_HA_i>\frac{\alpha}{2}+\frac{m}{2}+\frac{2u-1}{2}$,
then $\dim_H\left(\Delta_\Phi(A)\right)\geq u.$
    \item[(ii)] If
    $\sum\limits_{i=1}^k
\dim_HA_i>\frac{\alpha}{2}+\frac{m}{2}+\frac{1}{2}$,
then $\calL^1\left(\Delta_\Phi(A)\right)>0$.
\item[(iii)] If  
    $\sum\limits_{i=1}^k
\dim_HA_i>\frac{\alpha}{2}+\frac{m}{2}+\frac{3}{2}$,
then $\Int\left(\Delta_\Phi(A)\right)\neq \emptyset$.
\end{itemize}
\end{theorem}
\begin{remark}
    Observe that in the above result, since $\sum\limits_{i=1}^k
\dim_HA_i\leq \dim\Omega=n$, for $u>0$, condition $\sum\limits_{i=1}^k
\dim_HA_i>s(E,F)+\frac{m}{2}+\frac{2u-1}{2}$ implies that 
\begin{align*}
    m< \min\left\{  n_E+n_F, 2n-(n_E+n_F)\right\}-2u+1.
\end{align*}
In particular, when $u=1$, this becomes $m\leq \min\left\{  n_E+n_F, 2n-(n_E+n_F)\right\}-2$.
\end{remark}
\begin{proof}
Let $\Phi\in C^\infty(\Omega,\R)$ be a smooth function on $\Omega$. Assume that $\alpha\geq n$, and $m=0$; the case $m\geq 1$ can be proved by the same argument. 

For $E,F\subset \left\{ 1,\dots, k\right\}$, define\[W_{E,F}=\left\{ (x,y)\in Z: \Phi \text{ is }\Gamma_0(E,F) \text{ nondegenerate at } (x,y)\right\}.\]
By the assumptions, $\left\{ W_{E,F}\right\}$ forms a cover of $Z$, or $Z\subset \bigcup\limits_{E,F}W_{E,F}$.
Note that there is a finite number of pairs $\left\{ E,F\right\}$ such that $n_{E,F}\leq \alpha$. 

Fix such a pair $\left\{ E,F\right\}$ for now and assume that $W_{E,F}\neq \emptyset$. For convenience, we put $X=\Omega\times \Omega$, and for $x\in X$, write $x=(x',x'')$, where $x',x''\in \Omega$. Then we redefine $\Psi:X\to \R$ by $\Psi(x):=\Psi(x',x'')= \Phi^1(x')-\Phi^2(x'')$, for all $(x',x'')\in X$, where $\Phi^1=\Phi^2=\Phi$. The manifold $Z$ can be rewritten as $Z=\left\{   (x',x'')\in X: \Psi(x',x'')=0\right\}$. 

Let $\phi : X\times \R\setminus \left\{ 0\right\} \to \R$ be a smooth phase function defined by $\phi(x',x'',\theta)=\theta \cdot \Psi(x',x'')$, for $(x',x'',\theta)\in X\times (\R\setminus \left\{ 0\right\})$. One sees that $\phi$ parametrizes the canonical relation
\begin{align*}
    C_\phi
    &:=\left\{ (x',\theta d_{x'}\Phi^1(x');x'',\theta d_{x''} \Phi^2(x'')): \Phi^1(x')=\Phi^2(x''),\theta\neq 0\right\}\subset (T^\ast \Omega\setminus\bzero)\times(T^\ast \Omega\setminus\bzero).
\end{align*}
 Put $\sigma_L=E\cup F$, and $\sigma_R=E'\cup F'$, then $\sigma=(\sigma_L\mid \sigma_R)\in \mathcal{P}_{2k}$ is a nontrivial partition of $\left\{ 1,\dots,2k\right\}$. For $x=(x',x'')\in X$, the coordinate-partitioned version of $x$ with respect to $\sigma$ is denoted by $x=(x_L,x_R)$, where
\begin{align*}
   x_L=(x_L',x_L'') \in X_L\quad\text{ and } \quad x_R=(x_{R}',x_{R}'')\in X_R,
\end{align*}
with corresponding partitions $x'=(x_L',x_R')$ and $x''=(x_L'',x_R'')$. Note that, without loss of generality, we can assume that $n_E+n_F\geq n$, since if not, we can choose $E$ to be $E'$, and $F$ to be $F'$. It follows that $\dim X_L\geq \dim X_R$.

The coordinate partitioned version of $Z$ is denoted by $Z^\sigma=\left\{ (x_L,x_R): \Psi(x_L,x_R)=0\right\}$, and its conormal bundle is
\begin{align*}
    N^\ast Z^\sigma
    &=\left\{ (x_L,x_R,\theta D\Psi(x_L,x_R)): \Phi^1(x_L',x_R')-\Phi^2(x_L'',x_R'')=0,\theta\in \R\setminus \left\{ 0\right\}\right\}.
\end{align*}
The canonical relation associated to $Z^\sigma$ is the twisted conormal bundle of $Z^\sigma$, i.e.
\begin{align*}
    \Lambda^\sigma:= ( N^\ast Z^\sigma)'
    &=\left\{ (x_L,\xi_L;x_R,\xi_R): (x_L,x_R)\in Z^\sigma,  (\xi_L,-\xi_R)\perp TZ^\sigma, \xi_L, \xi_R\neq \bzero\right\}\\
    &=\{  (x_L',\,x_L'',\, \theta d_{x_L'}\Phi^1,\,-\theta d_{x_L''}\Phi^2;\,x_R',\,x_R'', \,-\theta d_{x_R'}\Phi^1,\, \theta d_{x_R''}\Phi^2): \\
    &\hspace{6cm}\Phi^1(x_L',x_R')=\Phi^2(x_L'',x_R''),\theta\neq 0\}.
\end{align*}
Let $\pi_L:\Lambda^\sigma\to T^\ast X_L$ be the cotangent space projection
\begin{align*} 
    \pi_L(x_L,\xi_L,x_R,\xi_R)=(x_L,\xi_L),\quad \text{ for }(x_L,\xi_L,x_R,\xi_R)\in \Lambda^\sigma.
\end{align*}
In order to apply Theorem \ref{thm: general_phi_configurations}, one needs to compute $\corank(D\pi_L)$ on $\Lambda^\sigma$. Consider the smooth map $F: X_L\times X_R\times (\R\setminus \left\{ 0\right\})\to (T^\ast X_L\setminus\bzero)\times(T^\ast X_R\setminus\bzero)$ defined by
\begin{align*}
       F(x_L',x_L'',x_R',x_R'',\theta)&=\left(x_L',\,x_L'',\, \theta d_{x_L'}\Phi^1,\,-\theta d_{x_L''}\Phi^2;\,x_R',\,x_R'', \,-\theta d_{x_R'}\Phi^1,\, \theta d_{x_R''}\Phi^2\right).
\end{align*}
One computes
\begin{align*}
    DF=
    \begin{bmatrix}
        \,I_{x_L'} &0&0&0&0\,\\
        \,0 &I_{x_L''}&0&0&0\,\\
        \theta H_{x_L',x_L'}\Phi^1 &0&\theta H_{x_L',x_R'}\Phi^1&0& (\nabla_{x_L'}\Phi^1)^T\\
        0& -\theta H_{x_L'',x_L''}\Phi^2&0 &-\theta H_{x_L'',x_R''}\Phi^2&(-\nabla_{x_L''}\Phi^2)^T\, \\
        0&0&I_{x_R'}&0&0\,\\
        0&0&0& I_{x_R''}&0\,\\
        -\theta H_{x_R',x_L'}\Phi^1&0&-\theta H_{x_R',x_R'}\Phi^1&0&-(\nabla_{x_R'}\Phi^1)^T\,\\
        0& \theta H_{x_R'',x_L''}\Phi^2&0&\theta H_{x_R'',x_R''}\Phi^2&(\nabla_{x_R''}\Phi^2)^T
    \end{bmatrix}.
\end{align*}
Therefore, for $\omega\in \Lambda^\sigma$, the tangent space to $\Lambda^\sigma$ at $\omega$ is given by
\begin{align*}
    T_\omega \Lambda^\sigma
    &=\big\{ (u_L',u_L'', \theta H_{x_L',x_L'}\Phi^1\cdot u_L'+\theta H_{x_L',x_R'}\Phi^1\cdot u_R'+\tau (\nabla_{x_L'}\Phi^1)^T,\\
    &\hspace{1cm}-\theta H_{x_L'',x_L''}\Phi^2\cdot u_L''-\theta H_{x_L'',x_R''}\Phi^2\cdot u_R''-\tau (\nabla_{x_L''}\Phi^2)^T;\ast,\ast,\ast,\ast):(u_L,u_R,\theta)\in \R^{2n+1},\\
    &\hspace{3cm}  \nabla_{x_L'}\Phi^1\cdot u_L'+\nabla_{x_R'}\Phi^1\cdot u_{R}'-\nabla_{x_L''}\Phi^2\cdot u_L''-\nabla_{x_R''}\Phi^2\cdot u_R''=0\big\},
\end{align*}
where we have suppressed the entries as irrelevant for the analysis of $\pi_L$.
Using the above, to find $\Ker(D\pi_L)$, it suffices to solve the following system
\begin{align*}
    \begin{cases}
    u_L'&=0\\
    u_{L''}&=0\\
    \theta H_{x_L',x_L'}\Phi^1\cdot u_L'+\theta H_{x_L',x_R'}\Phi^1\cdot u_R'+\tau (\nabla_{x_L'}\Phi^1)^T&=0\\
    -\theta H_{x_L'',x_L''}\Phi^2\cdot u_L''-\theta H_{x_L'',x_R''}\Phi^2\cdot u_R''-\tau (\nabla_{x_L''}\Phi^2)^T&=0\\
    \nabla_{x_L'}\Phi^1\cdot u_L'+\nabla_{x_R'}\Phi^1\cdot u_{R}'-\nabla_{x_L''}\Phi^2\cdot u_L''-\nabla_{x_R''}\Phi^2\cdot u_R''&=0 
\end{cases}.
\end{align*}
This is equivalent to
\begin{align*}
    \begin{cases}
    u_L'&=0\\
    u_{L''}&=0\\
     H_{x_L',x_R'}\Phi^1\cdot u_R'+\tau/\theta (\nabla_{x_L'}\Phi^1)^T&=0\\
     H_{x_L'',x_R''}\Phi^2\cdot u_R''+\tau/\theta (\nabla_{x_L''}\Phi^2)^T&=0\\
    \nabla_{x_R'}\Phi^1\cdot u_{R}'-\nabla_{x_R''}\Phi^2\cdot u_R''&=0 
\end{cases}\quad \quad, \quad \quad\theta\neq 0,
\end{align*}
or, in other words, $(u_L',u_L'')=\bzero$ and
\begin{align*}
    \begin{bmatrix}
        H_{x_L',x_R'}\Phi^1&0&(\nabla_{x_L'}\Phi^1)^T\\
        0&H_{x_L'',x_R''}\Phi^2 &(\nabla_{x_L''}\Phi^2)^T\\
        \nabla_{x_R'}\Phi^1&-\nabla_{x_{R}''}\Phi^2&0
    \end{bmatrix}
    \cdot
    \begin{bmatrix}
        u_{R'}\\
        u_{R''}\\
        \tau
    \end{bmatrix}
    =\begin{bmatrix}
        0\\
        0\\
        0
    \end{bmatrix}.
\end{align*}
By the assumption that $\corank(J_{E,F}\Phi)=0$, the above matrix has maximal rank everywhere on $W_{E,F}$. Therefore, $\Ker(D\pi_L)=\bzero$, and $\corank(D\pi_L)=0$ everywhere on $\Lambda^\sigma\mid_{W_{E,F}}$, the restriction of $\Lambda^\sigma$ to $W_{E,F}$. This implies that for $u>0$, the operator $\calT^\sigma$, defined as in \eqref{eq: op_T_sigma} and localized to $W_{E,F}$, satisfies \eqref{eq: condition_FIO_beta_sigma} with no loss in derivatives, i.e.
\begin{align*}
     \calT^\sigma: L_s^2(X_R)\to L_{s-m_{\text{eff}}^\sigma}^2(X_L), \quad \forall s\in \R.
\end{align*}
Note that by assumptions and \eqref{eq: gradient_condition}, $\rank (D_{x_L}\Psi)=\rank (D_{x_R}\Psi)=1$ everywhere, it follows that the double fibration condition $(DF)_\sigma$ holds. 

Therefore, by the above analysis, one concludes that for every point $(x_0,y_0)\in Z$, there exist a partition $\sigma=(E\mid F)\in \mathcal{P}_{2k}$ such that $\max(d_{\sigma,L},d_{\sigma,R})\leq \alpha$, and a neighborhood $W_{E,F}$ of $(x_0,y_0)$ for which the operator $\calT^\sigma$, localized to $W_{E,F}$, satisfies \eqref{eq: condition_FIO_beta_sigma} with no loss in derivatives.
Applying Theorem \ref{thm: general_phi_configurations} with 
\begin{align*}
    p=1\,,\, \alpha\,,\, \beta=0\,,\,M=\Omega\,,\,
\end{align*}
one obtains the desired results. This completes the proof of the theorem.
\end{proof}

\subsubsection{\texorpdfstring{$2$}{2}-point configurations}
Now we assume that $k=2$ and $p=1$. Let $X\subset \R^{d_X}$ and $Y\subset \R^{d_{Y}}$ are open sets, where $d_X,d_Y\geq 1$, and put $\Omega=X\times Y$. Let $\Phi: X\times Y\to \R$ be a smooth submersion. We further assume that
\begin{align}\label{eq: gradient_condition_Phi_k=2}
    \nabla _x\Phi(x,y)\neq \bzero, \nabla_{y}\Phi(x,y)\neq \bzero, \quad \forall (x,y)\in \Omega.
\end{align}
Observe that in this case, there is only one choice for the pair $\left\{ E,F\right\}$, up to permutation, that is $E=\left\{ 1\right\}, F=\left\{ 2\right\}$. Hence, the matrix $J_{E,F}\Phi$ is simpler, and the nondegeneracy conditions $\Gamma_m(E,F)$ can be easily checked . 
For $(x,y,x',y')\in \Omega\times \Omega$, define a $(d_X+d_Y+1)\times (d_X+d_Y+1)$ matrix
\begin{align}\label{eq: matrix_J_2_points_explicit}
    \calJ_\Phi:=\begin{bmatrix}
    0&\nabla_{y}\Phi(x,y) &-\nabla_{x'}\Phi(x',y')\,\\
       (\nabla_{x}\Phi(x,y))^T &H_{x,y}\Phi(x,y)&0\,\\
       (\nabla_{y'}\Phi(x',y'))^T &0 &H_{y',x'}\Phi(x',y')\,
    \end{bmatrix},
\end{align}
where $H_{x,y}\Phi$ is the mixed Hessian matrix, i.e., 
\begin{align*}
  H_{x,y}\Phi:=\left[  \frac{\partial^2\Phi}{\partial x_i\partial y_j}\right]_{\substack{1\leq i\leq d_X \\ 1\leq j\leq d_Y}}.
\end{align*} 
 If $\Phi$ is $\Gamma_m(E,F)-$nondegenerate for some integer $m\geq 0$ (see Definition \ref{def: Gamma_m(E,F)_nondegenerate}), we simply say $\Phi$ is $\Gamma_m-$nondegenerate.
 
Given subsets $A\subset X$, $B\subset Y$, define 
 $\Delta_\Phi(A,B)=\left\{  \Phi(x,y): x\in A,y\in B\right\}$.
 As an immediate consequence of Theorem \ref{thm: k_point_explicit_p=1}, we obtain the following. 

\begin{theorem}\label{thm: 2_point_explicit_p=1}
    Let $\Phi\in C^\infty(X\times Y,\R)$ be a smooth function. Assume that $\Phi$ is $\Gamma_m-$nondegenerate, with some nonnegative integer $m\geq 0$. 

    Suppose $A\subset X$ and $B\subset Y$ are compact sets such that $\dim_HA,\dim_HB>0$.
\begin{itemize}
\item[(i)] For $\, 0<u<1$, if $\dim_HA+\dim_HB>\frac{d_X+d_Y}{2}+\frac{m}{2}+\frac{2u-1}{2}$,
then $\dim_H(\Delta_\Phi(A,B))\geq u.$
    \item[(ii)] If
    $\dim_HA+\dim_HB>\frac{d_X+d_Y}{2}+\frac{m}{2}+\frac{1}{2}$,
then $\calL^1(\Delta_\Phi(A,B))>0$.
\item[(iii)] If  
    $\dim_HA+\dim_HB>\frac{d_X+d_Y}{2}+\frac{m}{2}+\frac{3}{2}$,
then $\Int(\Delta_\Phi(A,B))\neq \emptyset$.
\end{itemize}
\end{theorem}
\begin{remark}
    In Theorem \ref{thm: 2_point_explicit_p=1} (ii), given a smooth function $\Phi: X\times Y\to \R$, the best dimensional threshold for the dimensions of the underlying sets $A\subset X$ and $B\subset Y$, which ensuring $\calL^1(\Delta_\Phi(A,B))>0$, that one can expect is
    \begin{align*}
        \dim_HA+\dim_HB>\frac{d_X+d_Y+1}{2}.
    \end{align*}
    This lower bound can be achieved if $\Phi$ is $\Gamma_0$-nondegenerate, i.e. $\rank\left(H_{x,y}\right)$ is maximal and $\det (\calJ_\Phi)\neq 0$ everywhere on the incidence relation set $Z$. This lower bound is generally sharp, as our theorem on distances between general sets and sets lying on hyperplanes is sharp. However, many configuration functions of interest fail to satisfy the nonvanishing determinant condition, i.e. $\det(\calJ_f)=0$ on $Z$, even when the rank of the mixed Hessian matrix $H_{x,y}\Phi$ is maximal. These include distance  $\Phi(x,y)=\vert x-y\vert$ and dot product $\Phi(x,y)=x\cdot y$.
\end{remark}
Observe that if $\Phi$ satisfies \eqref{eq: gradient_condition_Phi_k=2} and for $r\geq 1$,
    $\rank (H_{x,y})\geq r$ everywhere,
then one sees that $\rank (\calJ_\Phi)\geq 2r$, which implies $\corank(\calJ_\Phi)\leq d_X+d_Y+1-2r$. By Theorem \ref{thm: 2_point_explicit_p=1} (ii), if $A\subset X$ and $B\subset Y$ are compact sets such that 
\[\dim_HA+\dim_HB>\frac{d_X+d_Y}{2}+\frac{d_X+d_Y+1-2r}{2}+\frac{1}{2}=d_X+d_Y+1-r,\]
then $\calL^1(\Delta_\Phi(A,B))>0$. Thus, we have the following result, which restates Theorem \ref{thm: 2_point_explicit_p=1_rank_condition_intro}.
\begin{theorem}\label{thm: 2_point_explicit_p=1_rank_condition}
   Let $\Phi\in C^\infty(X\times Y,\R)$ be a smooth function. Assume that 
    \begin{align*}
        \nabla _x\Phi(x,y)\neq \bzero, \nabla_{y}\Phi(x,y)\neq \bzero, \quad \forall (x,y)\in X\times Y,
    \end{align*}
    and for some $r\geq 1$, 
    \begin{align*}
        \rank\left(H_{x,y}\right)\geq r, \quad \forall (x,y)\in X\times Y.
    \end{align*}
    Assume that $A\subset X$ and $B\subset Y$ are compact sets.
    \begin{itemize}
      \item[(i)] For any $0<u<1$, if $\dim_HA+\dim_HB>d_X+d_Y+u-r$,
then $\dim_H(\Delta_\Phi(A,B))\geq u.$
    \item[(ii)] If
    $\dim_HA+\dim_HB>d_X+d_Y+1-r$,
then $\calL^1(\Delta_\Phi(A,B))>0$.
\item[(iii)] If  
    $\dim_HA+\dim_HB>d_X+d_Y+2-r$,
then $\Int(\Delta_\Phi(A,B))\neq \emptyset$.
    \end{itemize}
\end{theorem}
\begin{remark}
    We note that Theorem \ref{thm: 2_point_explicit_p=1} (iii) and Theorem \ref{thm: 2_point_explicit_p=1_rank_condition} (iii) follow from the local partition optimization in \cite[Theorem 2.5]{GreenleafIosevichTaylor2024}, but we state in terms of explicit $\Gamma_m(E,F)$-nondegenerate condition \eqref{def: Gamma_m(E,F)_nondegenerate}.
\end{remark}

\section{Dimension expansion: the bivariate case}\label{sec: dimexpansion_2vars}
In this section, we will apply our general theorems to study the dimension expansion phenomenon of smooth functions of two variables and give an alternative proof for the Elekes-R\'onyai type theorem for real analytic functions established by Raz and Zahl\cite{Raz_Zahl_2024} when the underlying sets have large Hausdorff dimension.

We start by considering the dimension expansion problem for bivariate real smooth functions. Let $\Omega\subset \R^2$ be an open set, and let $f\in C^\infty(\Omega)$. For convenience, we denote $f_x=\partial_x f$, $f_y=\partial_y f$, and similarly for higher order derivatives of $f$. In the domain where $\partial_{xy}f\neq 0$, define a smooth function $\rho$ by
\begin{align*}
    \rho:=\frac{\partial_xf\partial_yf}{\partial_{xy}f}.
\end{align*}
Then define an \textit{auxiliary function} $\kappa(x,y)$ by
\begin{align*}
    \kappa(x,y):=(\nabla f\wedge \nabla \rho)=\nabla f\wedge\nabla\left(\frac{\partial_xf\partial_yf}{\partial_{xy}f}\right).
\end{align*}
As mentioned in the introduction, this auxiliary function was introduced in \cite{Raz_Zahl_2024} to measure the extent to which $f$ ``looks like'' a special form near the point $(x,y)$. The function $\kappa(x,y)$ is closely related to the Blaschke curvature, and if it is bounded away from $0$, then $f$ satisfies a discretized energy dispersion estimate, which the authors used to prove the main results. In our context, the auxiliary function $\kappa(x,y)$ also appears naturally. Motivated by this, for each smooth function $f$, we define the bad set
\begin{align*}
    \calZ_{2,f}:=\left\{\partial_xf=0\right\}\cup\left\{\partial_yf=0\right\}\cup\left\{\partial_{xy}f=0\right\}\cup\left\{\kappa(x,y)=\bzero\right\}.
\end{align*}
We remark that $\calZ_{2,f}$ is a closed set, but not necessarily a submanifold of $\Omega$, as we only assume that $f$ is smooth without any further nondegeneracy conditions; and the zero set of a smooth function can be a closed fractal set of arbitrary large Hausdorff dimension (see \cite{Whitney_extension_1934}). In general, while $\calZ_{2,f}$ is possibly highly irregular, $f$ is well behaved away from $\calZ_{2,f}$, allowing us to obtain dimension expansion results. 
\begin{theorem}\label{thm: functions_2vars_smooth_restate}
Let $\Omega\subset \R^2$ be an open set, and let $ f\in C^\infty(\Omega)$ be smooth. Assume that $A,B\subset \R$ are compact sets such that $\dim_HA,\dim_HB>0$ and $A\times B\subset \Omega\setminus \calZ_{2,f}$.
    \begin{itemize}
        \item[(i)] For $u\in (0,1)$, if $\dim_HA+\dim_HB>\frac{2}{3}+u$, then $\dim_H\Delta_ f(A,B)\geq u.$
        \item[(ii)] If $\dim_HA+\dim_HB> \frac{5}{3}$, then $\calL^1(\Delta_ f(A,B))>0$. 
    \end{itemize}
\end{theorem}
Before giving the proof of the above theorem, we note that if one applies
Theorem \ref{thm: k_point_explicit_p=1} or Theorem \ref{thm: 2_point_explicit_p=1} to functions of two variables, then the matrix $J_{E,F}$, defined in \ref{eq: matrix_J}, does not have maximal rank on the incidence relation $Z$. Thus, in view of Theorem \ref{thm: k_point_explicit_p=1} (ii), for example,  we will get the following condition 
\[\dim_HA+\dim_HB>\frac{2}{2}+\frac{1}{2}+\frac{1}{2}=2,\] which is vacuous.
Therefore, looking at the rank of the mixed Hessian is not sufficient, and we need the full strength of Theorem \ref{thm: general_phi_configurations}. We show that when $\kappa(x,y)$ is nonzero, the Fourier integral operator $\calT$, defined as in \eqref{eq: op_T_sigma}, is associated to a folding canonical relation. This allows us to exploit the best known $L^2-$based Sobolev estimates for $\mathcal{T}$.
\begin{proof}[Proof of Theorem \ref{thm: functions_2vars_smooth_restate}]
We will prove part (ii), part (i) can be shown by the same argument. Without loss of generality, assume that there are open intervals $X,Y\subset \R$ such that $X\times Y\subset \Omega\setminus \calZ_{2,f}$, and $A\subset X$, $B\subset Y$ are compact sets. 

From the assumptions of the theorem one has
\begin{align}\label{eq: condition_1_2vars_smooth}
    f_x,f_y, f_{xy}, \kappa \neq 0, \quad \forall(x,y)\in X\times Y.
\end{align}
It suffices to show that the conclusion of the theorem holds for the restriction of $f$ to $X\times Y$, which by abuse of notation, we still denote as $f$. 
Recall that, in view of Theorem \ref{thm: general_phi_configurations}, one needs to find the configuration's incidence relation $Z$, and then shows that
there exist constants $\alpha,\beta\geq 0$ such that at every point $(z,\xi)$ of the conormal bundle $N^\ast Z$, there exist a bipartite partition $\sigma=(\sigma_L\mid\sigma_R)$ of variables and a conic neighborhood $\mathcal{U}$ of $(z,\xi)$ for which the associated operator $\calT^\sigma$, microlocalized to $\mathcal{U}$, satisfies the condition $[(DF)_\sigma;\beta]$ \eqref{eq: condition_FIO_beta_sigma} with a loss of at most $\beta$ derivatives.

We define a smooth function $\Psi: X\times Y\times X\times Y\to \R$ by
\begin{align*}
    \Psi(x,y,x',y'):= f^1(x,y)- f^2(x',y'), \quad \forall (x,y),(x',y')\in X\times Y,
\end{align*}
where $ f^1= f^2= f$. Ordering the variables $x,y,x',y'$, in order $1,2,3,4$, and partitioning them by $\sigma=(\sigma_L\mid \sigma_R)=(14\mid 32)$. The choice of $\sigma$ was made in order to minimize $\alpha$ and $\beta$ when we apply Theorem \ref{thm: general_phi_configurations} (see \cite{GreenleafIosevichTaylor2024} for other examples of how to choose such optimized partitions). 

Denote $X_L=X_R=X\times Y$, we separate the variables $\{x,y,x',y'\}$ into two groups $\{x,y'\}$ and $\{x',y\}$ corresponding to $\sigma$ and write $ x_L=(x,y')\in X_L$, $x_R=(x',y)\in X_R$.

By \eqref{eq: condition_1_2vars_smooth}, the incidence relation $Z:=\Psi^{-1}(0)=\left\{(x,y';x',y):  f(x,y)= f(x',y')\right\}\subset X_L\times X_R$
is a smooth manifold of codimension $1$. The canonical relation related to our analysis is the twisted conormal bundle of $Z$, that is,
\begin{align*}
    \Lambda:= ( N^\ast Z)'
    &=\left\{ (x_L,\xi_L;x_R,\xi_R): (x_L,x_R)\in Z,  (\xi_L,-\xi_R)\perp TZ, \xi_L, \xi_R\neq \bzero\right\}\\
    &=\{  (x,y',\, \theta  f^1_x,\,-\theta  f^2_{y'};\,x', y,\, \theta  f^2_{x'},\,-\theta  f^1_{y}): f^1(x,y)= f^2(x',y'),\theta\neq 0\}.
\end{align*}
 In general, the nature of the singularities of $\pi_L$ and $\pi_R$ may be quite different, and this is reflected in the mapping properties of the related operators. We claim that under the assumptions \eqref{eq: condition_1_2vars_smooth}, we have the following.
 \begin{lemma}\label{lem: Lambda_whitney_fold}
     $\Lambda$ is a folding canonical relation, that is, at any degenerate points $\zeta_0\in \Lambda$, both cotangent space projections $\pi_L$ and $\pi_R$ have Whitney fold singularities.
 \end{lemma} 
 Assuming this lemma is true, then by Theorem \ref{thm: FIO_Whitney_singularities}, the operator $\calT^\sigma$, defined as in \eqref{eq: op_T_sigma} 
 satisfies
\begin{align*}
    \calT^\sigma: L_s^2(X_R)\to L^2_{s-m_{eff}-\frac{1}{6}}(X_L), \quad \text{ for all }s\in \R.
\end{align*}
Thus $\calT^\sigma$ satisfies $[(DF)_\sigma;\beta]$ with $\sigma=(14\,\vert \,32)$ and $\beta=1/6$. Thus, applying Theorem \ref{thm: general_phi_configurations} (ii) with $k=2$, $p=1$, $\alpha=2$, $\beta=1/6$, we conclude that $\calL^1(\Delta_ f(A,B))>0$ provided 
\begin{align*}
    \dim_H A+\dim_HB>\frac{2-1}{2}+1+\frac{1}{6}= \frac{5}{3}.
\end{align*}

For part (i), as mentioned in the beginning, the proof is similar. We apply Theorem \ref{thm: general_phi_configurations} (ii) with $k=2$, $p=1$, $\alpha=2$, $\beta=1/6$, and $u\in (0,1)$, we conclude that $\dim_H\Delta_ f(A,B)\geq u$ provided 
\begin{align*}
    \dim_H A+\dim_HB>\frac{2-1}{2}+u+\frac{1}{6}= \frac{2}{3}+u.
\end{align*}
This finishes the proof of the theorem.
\end{proof}
    Now we turn to the proof of Lemma \ref{lem: Lambda_whitney_fold}. We make the following observation related to the auxiliary function $\kappa$. Define a smooth map $P: X\times Y\to \R^2$, $(x,y)\mapsto (P_1(x,y), P_2(x,y))$, where
\begin{align}\label{eq: def_P_kappa}
    P_1(x,y)=f(x,y)\quad \text{ and }\quad P_2(x,y)=\rho(x,y), \quad \forall (x,y)\in X\times Y.
\end{align}
From definition $\kappa=\nabla f \wedge \nabla \rho$, and by conditions \eqref{eq: condition_1_2vars_smooth}, one sees that the differential $dP$ has nonzero determinant everywhere in $X\times Y$, and thus the map $P$ is a local diffeomorphism at every point. More precisely, we have the following lemma.
\begin{lemma}\label{lem: P_local_diff}
    Let $X,Y\subset \R$ be open intervals. Let $f: X\times Y\to \R$ be a smooth function such that $f_x,f_y,f_{xy},\kappa\neq 0$ for all $(x,y)\in X\times Y$. Then for each $(x_0,y_0)\in X\times Y$, there exists a neighborhood $W$ of $(x_0,y_0)$ such that the restriction of $P=(f,\rho)$ to $W$, denoted by $P\vert_W: W\to P(W)$, is a diffeomorphism.
\end{lemma}
\begin{proof}[Proof of Lemma \ref{lem: Lambda_whitney_fold}]
 By symmetry, it suffices to show that at every point of the canonical relation $\Lambda$, the projection $\pi_L$ is a Whitney fold.

To this end, we will show that the set of singular points $\Sigma$ where $D\pi_L$ is degenerate is a smooth surface in $\Lambda$ of codimension $1$, the determinant $\det( D\pi_L)$ vanishes to first order on $\Sigma$ (i.e. its gradient is nonzero), and at every point $\zeta\in \Sigma$, the kernel of $D\pi_L$ is transverse with the tangent space to $\Sigma$ at the based point $\zeta$, that is,
\begin{align}\label{eq: def_fold_singularity}
    \nabla(\det(D\pi_L))_{\zeta} \neq 0 \quad \text{ and }\quad T_{\zeta}\Sigma\,\,+\,\,\Ker(D\pi_L)_{\zeta}\,\,=\,\,T_{\zeta}\Lambda\quad \text{ for }\zeta\in \Sigma;
\end{align}
(see, for example, \cite[Chapter III.4]{Golubitsky_Guillemin_book_1973}).

It suffices to show \eqref{eq: def_fold_singularity} in local coordinates at each point of $\Lambda$. Let $\zeta_0\in \Lambda$, then $\zeta_0$ corresponds to $z_0=(x_0,y_0',x_0', y_0)\in Z$ and $\theta_0\in \R\setminus \{0\}$, i.e.,
\begin{align*}
    \zeta_0=\left(x_0,y_0',\, \theta_0  f^1_x(x_0,y_0),\,-\theta_0  f^2_{y'}(x_0',y_0')\,;\,x_0', y_0,\, \theta_0  f^2_{x'}(x_0',y_0'),\,-\theta_0  f^1_{y}(x_0,y_0)\right) \in \Lambda.
\end{align*}
To find a local coordinates at $\zeta_0$, we consider the equation $\Psi(x,y,x',y')= f^1(x,y)- f^2(x',y')=0$ near $z_0$. By \eqref{eq: condition_1_2vars_smooth}, one has $\partial_{y}\Psi(z_0)=\partial_{y} f^1(x_0,y_0)\neq 0$, the implicit function theorem implies that there exist an open set $W=W_{x_0}\times   W_{y_0'}\times W_{x_0'}$, where
$W_{x_0},W_{y_0'}$ and $W_{x_0'}$ are neighborhoods of $x_0, y_0'$, and $x_0'$ respectively, and a unique differentiable function $\varphi:W\to \R$ such that $y_0=\varphi(x_0,y_0',x_0')$ and 
\begin{align}\label{eq: Z_locally_graph_varphi}
    \Psi\left(x,\varphi(x,y',x'),x',y'\right)= f^1(x,\varphi(x,y',x'))- f^2(x',y')=0, \quad \forall (x,y',x')\in W.
\end{align}
Moreover, the partial derivatives of $\varphi$ are given by
\begin{align}\label{eq: derivative_implicit_function}
    \varphi_x=-\frac{ f^1_x}{ f^1_{y}},\quad \varphi_{y'}=\frac{ f^2_{y'}}{ f^1_{y}}, \quad\varphi_{x'}=\frac{ f^2_{x'}}{ f^1_{y}}, \quad\forall (x,y',x')\in W.
\end{align}
Thus, near $z_0$, $Z$ is locally the graph of $\varphi:W\to \R$, denoted by $Z_\varphi$. Let $\mathcal{U}_{\zeta_0}$ be the twisted conormal bundle of $Z_\varphi$. Without loss of generality, we assume $\theta_0>0$, and consider $\mathcal{U}_{\zeta_0}^+=\mathcal{U}_{\zeta_0}\cap \{\theta>0\}$ as a conic neighborhood of $\zeta_0$; and by abuse of notation, we also refer to it as $\mathcal{U}_{\zeta_0}$. Then $\mathcal{U}_{\zeta_0}$ can be parametrized by $F:W\times (0,\infty)\to (T^\ast X_L\setminus \bzero)\times (T^\ast X_R\setminus \bzero)$, defined by
\begin{align*}
      F(x,y',x',\theta)&=\big(x,y',\, \theta  f^1_{x}(x,\varphi(\cdot)),\,-\theta  f^2_{y'}(x',y')\,;\,x',\varphi(\cdot),\, \theta  f^2_{x'}(x',y'),\,-\theta  f^1_{y}(x,\varphi(\cdot))\big),
\end{align*}
for $(x, y',x',\theta)\in W\times (0,\infty)$.

Let $\pi_L:\Lambda\to T^\ast X_L$ be the cotangent space projection, defined by
\begin{align*} 
    \pi_L(x_L,\xi_L,x_R,\xi_R)=(x_L,\xi_L),\quad \text{ for }(x_L,\xi_L,x_R,\xi_R)\in \Lambda.
\end{align*}
Then, the composite mapping $g=\pi_L\circ F: W\times (0,\infty)\to T^\ast X_L$ is a local representation of $\pi_L$, given by
\begin{align*}
    (x,y',x',\theta)\mapsto \big(x,y',\, \theta  f^1_{x}\left(x,\varphi(x,y',x')\right),\,-\theta  f^2_{y'}(x',y')\big).
\end{align*}
It suffices to show that $g$ has Whitney fold singularities. In other words, we will verify \eqref{eq: def_fold_singularity} with $g$ in place of $\pi_L$, and $\Sigma(g)$ in place of $\Sigma$, where $\Sigma(g)$ is the set of critical points of $g$, i.e.,
\begin{align}\label{eq: critial_set_g}
    \Sigma(g):=\left\{(x,y',x',\theta): \det(Dg)(x,y',x',\theta)=0\right\}.
\end{align}
One computes
\begin{align}\label{eq: differential_Dg}
    Dg=
    \begin{bmatrix}
         \,I_{x} &0&0&0\,\\
        \,0 &I_{y'}&0&0\,\\
        \theta\left(  f^1_{xx}+ f^1_{xy}\cdot\varphi_x\right) &0&\theta  f^1_{xy}\cdot\varphi_{x'}&  f^1_{x}\\
        -0& -\theta  f^2_{y'y'}&- \theta f^2_{x'y'}&- f^2_{y'}\,
    \end{bmatrix},
\end{align}
which has determinant equal to
\begin{align*}
    \det(Dg)&=\theta\left( f^2_{x'y'} f^1_x- f^1_{xy}\varphi_{x'} f^2_{y'}\right).
\end{align*}
Plugging $\varphi_{x'}=\frac{ f^2_{x'}}{ f^1_{y}}$ from \eqref{eq: derivative_implicit_function} into the above expression, one gets
\begin{align*}
    \det(Dg)
    &=\theta\left( f^2_{x'y'} f^1_x- f^1_{xy}\frac{ f^2_{x'}}{ f^1_{y}} f^2_{y'}\right)=\theta\frac{ f^1_{xy} f^2_{x'y'}}{ f^1_y}\left(\frac{ f^1_x f^1_y}{ f^1_{xy}}-\frac{ f^2_{x'} f^2_{y'}}{ f^2_{x'y'}}\right).
\end{align*}
Since $f^1=f^2=f$, and $\rho=\frac{f_xf_y}{f_{xy}}$, the above formula gives
\begin{align}\label{eq: determinant_Dg}
    \det(Dg)(x,y',x',\theta)=\theta \frac{f_{xy}\left(x,\varphi(x,y',x')\right)\, f_{xy}(x',y')}{f_y(x,\varphi(x,y',x'))}\cdot\Big[\rho\left(x,\varphi(x,y',x')\right)-\rho(x',y')\Big],
\end{align}
for $(x, y',x',\theta)\in W\times (0,\infty)$. By the assumptions \eqref{eq: condition_1_2vars_smooth}, $\det(Dg)$ vanishes if and only if 
\begin{align*}
    \rho\left(x,\varphi(x,y',x')\right)-\rho(x',y')=0.
\end{align*}
Noting that each point $(x,y',x',\theta)\in \Sigma(g)$ satisfies the following system of equations
\begin{align*}
    \begin{cases}
        f(x,\varphi(x,y',x'))-f(x',y')&=0\\
        \det(Dg)(x,y',x',\theta)&=0
    \end{cases}\quad,
\end{align*}
where the first follows from \eqref{eq: Z_locally_graph_varphi} and the second follows from the definition of $\Sigma(g)$. This is equivalent to
\begin{align}\label{eq: critical_point_and_P}
    \begin{cases}
        f(x,\varphi(x,y',x'))=f(x',y')\\
        \rho(x,\varphi(x,y',x'))=\rho(x',y')
    \end{cases} \quad \text{ or }\quad P\left(x,\varphi(x,y',x')\right)=P(x',y'),
\end{align}
where $P=(f,\rho)$ is defined in \eqref{eq: def_P_kappa}. Recall that by Lemma \ref{lem: P_local_diff}, restricting $W=W_{x_0}\times W_{y_0'}\times W_{x_0'}$ if necessary, we can assume that $P\vert_W$ is a diffeomorphism onto its image. Combining this with \eqref{eq: critical_point_and_P}, we conclude that $(x,y',x',\theta)\in \Sigma(g)$ if and only if $x'=x$, $y'=\varphi(x,y',x')$, i.e., the critical set of $g$, given by
\begin{align}
    \Sigma(g)=\left\{(x,y',x,\theta): x\in W_{x_0}, y'\in W_{y_0'}, \theta\in (0,\infty)\right\}
\end{align}
is a submanifold of codimension $1$.

Next, from \eqref{eq: determinant_Dg}, we compute
\begin{align*}
    \partial_{x}(\det(Dg))&=\theta \left(\ast\right)\cdot \Big[\rho\left(x,\varphi(\cdot)\right)-\rho(x',y')\Big]\\
    & \hspace{2cm}+\theta \frac{f_{xy}\left(x,\varphi(\cdot)\right)\, f_{xy}(x',y')}{f_y(x,\varphi(\cdot))}\cdot\Big[\rho_x(x,\varphi(\cdot))+\rho_y(x,\varphi(\cdot))\varphi_x(\cdot)\Big],
\end{align*}
where we have suppressed the term as irrelevant for the analysis by $\ast$.
Plugging $\varphi_x=-\frac{f_x}{f_y}$ from \eqref{eq: derivative_implicit_function} into the above expression, at each point $\omega=(\tilde x,\tilde y,\tilde {x},\tilde\theta)\in \Sigma(g)$, we see that
\begin{align*}
    \partial_{x}\left(\det(Dg)\right)\bigg\vert_{\omega=(\tilde x,\tilde y,\tilde {x},\tilde\theta)} &= \tilde\theta \frac{ [f_{xy}(\tilde x,\tilde y)]^2}{f_y(\tilde x,\tilde y)}\cdot\left[\rho_x(\tilde x,\tilde{y}-\rho_y(\tilde x,\tilde{y})\frac{f_x(\tilde x,\tilde{y})}{f_y(\tilde x,\tilde{y})}\right]\\
    &=\tilde\theta\left[ \frac{f_{xy}\left(\tilde x,\tilde{y}\right)}{f_y(\tilde x,\tilde{y})}\right]^2\cdot\left[\rho_xf_y-\rho_yf_x\right](\tilde x,\tilde{y})\\
    &=-\tilde\theta \left[ \frac{f_{xy}\left(\tilde x,\tilde{y}\right)}{f_y(\tilde x,\tilde{y})}\right]^2\cdot\kappa(\tilde x,\tilde{y}) \\
    &\neq 0,
\end{align*}
where in the last line we used the assumptions \eqref{eq: condition_1_2vars_smooth}. This yields that
\begin{align}\label{eq: nonzero_determinant_g}
    \nabla(\det(Dg))({\omega})\neq 0 \quad \text{ for all }\omega\in \Sigma(g).
\end{align}
For the next step, we fix a point $\omega=(\tilde x,\tilde y,\tilde x,\tilde \theta)\in \Sigma(g)$, then the tangent space to $\Sigma(g)$ at $\omega$ is
\begin{align}\label{eq: tangent_space_critical_set_case_1}
    T_{\omega}\Sigma(g)=\left\{(v_{x'}, v_{y'},v_{x'},\tau): v_{x'},v_{y'},\tau\in \R\right\}\subset T_\omega (W\times (0,\infty))=\R^4.
\end{align}
To find $\Ker(Dg)$, one solves $Dg(v_x,v_y',v_{x'}, \tau)=0$, where $(v_x,v_y',v_{x'},\tau)\in  T_\omega (W\times (0,\infty))$. From \eqref{eq: differential_Dg}, one has
\begin{align*}
    \begin{cases}
        v_x&=0\\
        v_{y'}&=0\\
       \theta( f^1_{xx}+ f^1_{xy}\varphi_x)\cdot v_x+\theta  f^1_{xy}\varphi_{x'}\cdot v_{x'}+ f^1_{x}\cdot \tau &=0\\
       \theta f^2_{y'y'}\cdot v_{y'}+\theta  f^2_{x'y'}\cdot v_{x'}+ f^2_{y'}\cdot \tau &=0
    \end{cases}\quad.
\end{align*}
This is equivalent to
\begin{align*}
       \begin{cases}
        v_x,v_{y'}&=0\\
        \theta  f^1_{xy}\varphi_{x'}\cdot v_{x'}+ f^1_{x}\cdot \tau&=0\\
        \theta  f^2_{x'y'}\cdot v_{x'}+ f^2_{y'}\cdot \tau&=0
    \end{cases}
    \quad \text{ or }\quad   \begin{cases}
         v_x,v_{y'}&=0\\
        v_{x'}&=\frac{-\tau}{\theta}\frac{ f^2_{y'}}{ f^2_{x'y'}}
    \end{cases}\quad,\quad \theta \neq 0.
\end{align*}
Therefore, we have
\begin{align}\label{eq: kernel_case_1}
    \Ker(Dg)_\omega= \left\{\left(0,0,\frac{ f_{y}(\tilde x,\tilde y)}{ f_{xy}(\tilde x,\tilde y)}\tau, \tau\right): \tau\in \R\right\}.
\end{align}
From \eqref{eq: tangent_space_critical_set_case_1} and \eqref{eq: kernel_case_1}, using the fact that $ f_{y}/ f_{xy}\neq 0$, one sees that
\begin{align}\label{eq: transverse_kernel_tangent}
    T_{\omega}\Sigma(g)+\Ker(Dg)_\omega= T_\omega (W\times (0,\infty))=\R^4,
\end{align}
i.e., $\Ker(Dg)_\omega$ intersects $T_{\omega}\Sigma(g)$ transversally. Combining \eqref{eq: transverse_kernel_tangent} with \eqref{eq: nonzero_determinant_g}, we conclude that $\omega\in \Sigma(g)$ is a fold point. Since this holds for every point in $\Sigma(g)$, we conclude that $g$ has Whitney fold singularities with the folding surface $\Sigma(g)$. Since $g$ is a local representation for $\pi_L\vert_{\mathcal{U}_{\zeta_0}}$, this yields that $\pi_L\vert_{\mathcal{U}_{\zeta_0}}$ has Whitney fold singularities for each point $\zeta_0\in \Lambda$. This completes the proof of the lemma.
\end{proof}

Next, we will present the proof of Theorem \ref{thm: Elekes_Ronyai_analytic_2vars}, namely the Elekes-R\'onyai theorem for bivariate real analytic functions. The auxiliary function $\kappa$ is useful for the following reason (see also \cite{Elekes_Ronyai_2000,Raz_Sharir_2017}).
\begin{lemma}[\cite{Raz_Zahl_2024} Lemma 3.3]\label{lem: K_Phi_special_form}
Let $\Omega\subset \R^2$ be a connected open set and let $ f: \Omega\to \R$ be analytic. Suppose that none of $f_x$, $ f_y$, and $ f_{xy}$ vanishes identically on $\Omega$. Then
    $\kappa(x,y)=  \nabla f\wedge\nabla\rho$
    vanishes identically on $\Omega$ if and only if $f$ is an analytic special form.
\end{lemma}
Here we restate Theorem \ref{thm: Elekes_Ronyai_analytic_2vars} for the reader's convenience.
\begin{theorem}
    Let $\Omega\subset \R^2$ be a connected open set, and let $f\in C^\omega(\Omega)$ be a bivariate analytic function. Then either $f$ is an analytic special form, or for every pair of compact sets $A,B\subset \R$ with positive Hausdorff dimension such that $A\times B\subset \Omega$, we have the following
    \begin{itemize}
        \item[(i)] For $u\in (0,1)$, if $\dim_HA+\dim_HB>\frac{2}{3}+u$, then $\dim_H\Delta_f(A,B)\geq u$.
        \item[(ii)] If $\dim_HA+\dim_HB>\frac{5}{3}$, then $\calL^1(\Delta_f(A,B))>0$. 
    \end{itemize}
\end{theorem}
\begin{proof}
We will prove part (ii), part (i) follows by the same argument. Suppose that $f:\Omega\to \R$ is an analytic function that is not an analytic special form. Assume that $A\subset X$ and $B\subset Y$ are compact sets such that
\begin{align*}
    \dim_HA+\dim_HB>\frac{5}{3},
\end{align*}
where $X,Y\subset \R$ are closed intervals such that $X\times Y\subset \Omega$.
Then by Lemma \ref{lem: K_Phi_special_form}, we have that
\begin{align}\label{eq: nonequiv_zero_grad_Phi}
    f_x \not\equiv 0, f_y \not\equiv 0, f_{xy}\not\equiv 0,\text{ and } \kappa(x,y)\not\equiv 0,\quad \forall (x,y)\in X\times Y.
\end{align}
Recall that we define
\begin{align*}
    \calZ_{2,f}:=\left\{f_x=0\right\}\cup\left\{f_y=0\right\}\cup\left\{f_{xy}=0\right\}\cup\left\{\kappa=0\right\}.
\end{align*}
Since $f$ is analytic, $\calZ_{2,f}$ is a finite union of real analytic varieties of dimension $\dim \calZ_{2,f}\leq 1$, in fact it is a semi-analytic set \cite{Lojassiewicz_1965}. Let $\delta=\frac{1}{2}(\dim_HA+\dim_HB-5/3)>0$, then by Lemma \ref{lem: local_dim_away_analytic_variety}, there exists a point $(a,b)\in \Omega\setminus \calZ_{2,f}$ such that
\begin{align}\label{eq: local_dim_AB}
    \dim_H(A\cap \tilde X) \geq \dim_HA-\delta/4, \quad\text{and}\quad
    \dim_H(B\cap \tilde Y)\geq \dim_HB-\delta/4,
\end{align}
for all neighborhoods $\tilde X$ and $\tilde Y$ of $a$ and $b$, respectively. By choosing sufficiently small intervals $\tilde X\subset X$ and $\tilde Y\subset Y$, letting $\tilde A= A\cap \tilde X$ and $\tilde B=B\cap \tilde Y$, we can assume that 
\begin{align*}
    \tilde A\times\tilde B \subset \Omega \setminus &\calZ_{2,f},\quad  \text{ and }\\
   \dim_H\tilde A+\dim_H\tilde B&>\dim_HA+\dim_HB-\delta/2>\frac{5}{3},
\end{align*}
where the last line follows from \eqref{eq: local_dim_AB}. Applying Theorem \ref{thm: functions_2vars_smooth_restate} with $f$, $\tilde A$ and $\tilde B$, we conclude that $\calL^1\left(\Delta_f(\tilde A,\tilde B)\right)>0$, which yields $\calL^1\left(\Delta_f( A, B)\right)>0$ as desired. This finishes the proof of the theorem.
\end{proof}

\section{Dimension expansion: the trivariate case}\label{sec: dimexpansion_3vars}
We will discuss our result on the dimension expansion phenomenon of smooth functions of three variables. As a corollary, we prove an Elekes-R\'onyai type theorem for trivariate real analytic functions.

Similarly to the bivariate case, we shall first discuss the dimension expansion results for trivariate smooth functions.
Let $\Omega\subset \R^3$ be an open set. Define, for $x=(x_1,x_2,x_3)\in \Omega$, 
\begin{align}\label{eq: def_gi_f}
    \calg_1(x)&=f_3(x)f_{12}(x)-f_{13}(x)f_2(x), \notag\\
    \calg_2(x)&=f_3(x)f_{12}(x)-f_{23}(x)f_1(x), \\
\calg_3(x)&=f_1(x)f_{23}(x)-f_{13}(x)f_2(x).   \notag
\end{align}
where $f_i(x)=\partial_{x_i}f(x)$, and $f_{ij}(x)=\frac{\partial^2f}{\partial x_i\partial x_j}(x)$, $1\leq i,j \leq 3$.
The common zero set of these functions is then denoted by $\calZ_{3,f}$, i.e.
\[\calZ_{3,f}:=\bigcap_{i=1}^3 \calg_i^{-1}(0)=\left\{ x\in \Omega:\calg_i(x)=0,\forall 1\leq i \leq 3\right\}.\]
It is easy to see that $\calZ_{3,f}\supset \left\{ x:\nabla f(x)=0\right\}$, the critical set of $f$.
 Observe that $\left\{ \calg_i\right\}_i$ are smooth functions, and if $\calg_i=\calg_j=0$, for some $1\leq i<j\leq 3$, then $\calg_1=\calg_2=\calg_3=0$.

Given compact sets $A,B,C\subset \R$, $A\times B\times C\subset \Omega$, recall 
\[\Delta_f(A,B,C):=\left\{ f(x_1,x_2,x_3): x_1\in A,x_2\in B,x_3\in C\right\}.\] 
\begin{theorem}\label{thm: functions_3vars_smooth_restate}
    Let $\Omega\subset \R^3$ be an open set, and let $f\in C^\infty(\Omega,\R)$. Let $A_i\subset \R$, $1\leq i\leq 3$, be compact sets such that each $\dim_HA_i>0$, and $A_1\times A_2\times A_3\subset \calZ_{3,f}^c$.

    \begin{itemize}
        \item[(i)] If $0<u<1$, and $\sum\limits_{i=1}^3\dim_HA_i>1+u,$ then $\dim_H\Delta_f(A_1,A_2,A_3)\geq u$.
        \item[(ii)] If $\sum\limits_{i=1}^3\dim_HA_i>2$, then $\calL^1\left(\Delta_f(A_1,A_2,A_3)\right)>0$.
    \end{itemize}
     Moreover, if there exists $i_0\in \left\{ 1,2,3\right\}$ such that $\nabla \calg_{i_0}(x)\neq 0$, for all $x\in \Omega$, then (i) holds with the assumption that $A_1\times A_2\times A_3\subset \Omega$. 
\end{theorem}
\begin{proof}[Proof of Theorem \ref{thm: functions_3vars_smooth_restate}] We will prove part (ii), part (i) can be shown by the same argument. To begin with, let $\Omega\subset \R^3$ be an open set, and let $f:\Omega\to \R$ be a smooth function. Assume $A_i\subset \R$, $1\leq i\leq 3$, are compact sets such that $A=A_1\times A_2\times A_3\subset \calZ_{3,f}^c$.

To show that $\calL^1\left(\Delta_f(A_1,A_2,A_3)\right)>0$, in view of Theorem \ref{thm: k_point_explicit_p=1}, we first need to find the configuration's incidence relation $Z$, and show that at every point $z\in Z$, there exist nonempty subsets $E,F$ of $\{1,2,3\}$ such that $f$ is $\Gamma_m(E,F)-$nondegenerate at $z$, for some $m\geq 0$.

We claim that there exist $i_0\in \left\{ 1,2,3\right\}$, and an open cube $U=U_1\times U_2\times U_3\subset \calZ_{3,f}^c$, such that 
\begin{align}\label{eq: restriction_f_U}
    \dim_H(A\cap U)=\dim_HA, \quad \text{ and }\quad \calg_{i_0}(x)\neq 0, \quad\forall x\in U.
\end{align}
To see this, without loss of generality, we can assume that $\Omega$ is bounded and hence  $\calg_i^{-1}(0)$, and $\calZ_{3,f}$ are compact. For each $i$, put $W_i=\Omega\setminus g^{-1}(0)$, then $A\subset \bigcup\limits_{i=1}^3 W_i$. By pigeonholing, there exists $i_0$ such that $\dim_HA\cap W_{i_0}=\dim_HA$.
Since $W_{i_0}$ is open, we can write it as a countable union of open balls, $W_{i_0}=\bigcup\limits_{j=1}^\infty B_j$, where for each $j\geq 0$, the ball $B_j(w_j,r_j)$ centers at $w_j\in W_{i_0}$, and has radius $r_j<\frac{1}{4}d(w_j,g^{-1}(0))$. 
Again, by pigeonholing, there exists a ball $B_{j_0}(w_{j_0},r_{j_0})$ such that
\begin{align*}
    \dim_H(A\cap B_{j_0})=\dim_HA.
\end{align*}
Let $U=U_1\times U_2\times U_3\subset \R^3$ be an open cube of side length $2r_{j_0}$, and centers at $w_{j_0}$. It is easy to see that\[B_{j_0}\subset U\subset W_{i_0}\subset \calZ_{3,f}^c,\]
and
\begin{align*}
    \dim_H (A\cap U)\geq \dim_H  (A\cap B_{j_0})=\dim_HA,
\end{align*}
as desired. Set $A_i'=A_i\cap U_i$, $1\leq i\leq 3$, one has $\sum\limits_{i=1}^3\dim_HA_i'=\sum\limits_{i=1}^3\dim_HA_i$, and $\calg_{i_0}\neq 0$ everywhere on $U$.

\begin{sloppypar}
Next, consider all pairs of nonempty subsets $E,F\subset \left\{ 1,2,3\right\}$ such that $F=E^c=\left\{ 1,2,3\right\}\setminus E$, and $\vert E\vert\geq \vert F\vert$. There are three cases: 
\begin{align*}
     (E;F)=(2\,3;1)\,,\quad(E;F)=(1\,2;3)\,,\quad (E;F)=(1\,3;2)\,.
\end{align*}    
The reason we choose $E,F$ satisfying the above conditions is to obtain the optimal dimensional threshold when we apply Theorem \ref{thm: k_point_explicit_p=1}. By definition \eqref{eq: matrix_J}, since $F=E^c=E'$, the matrix $J_{E,F}f(x,y)$, for $(x,y)\in \Omega\times \Omega$, is defined by
\begin{align*}
        J_{E,F}f:=\begin{bmatrix}
    0&\nabla_{x_{F}}f(x) &-\nabla_{y_{E}}f(y)\,\\
       (\nabla_{x_E}f(x))^T &H_{x_E,x_{F}}f(x)&0\,\\
       (\nabla_{y_F}f(y))^T &0 &H_{y_F,y_{E}}f(y)\,
    \end{bmatrix}.
\end{align*} 
\end{sloppypar}
If $(E;F)=(2\,3;1)$, then the matrix $J_{E;F}f$ becomes
\begin{align}\label{eq: matrix_J_1}
    J_{1}f=
    \begin{bmatrix}
    0&f_1(x)& -f_2(y)&-f_3(y) \\
    f_2(x)&f_{12}(x)& 0 &0\,\\
   f_3(x)& f_{13}(x)& 0& 0\\
    f_1(y)& 0& f_{12}(y)&f_{13}(y)
    \end{bmatrix},
\end{align}
which has determinant $Q_{1}=\det[J_{1}f]$,
\begin{align}\label{eq: determinant_3var_Q1}
   Q_{1}
   &=-\big[f_3(x)f_{12}(x)-f_{13}(x)f_2(x)\big]\big[f_3(y)f_{12}(y)-f_{13}(y)f_2(y)\big]=-\calg_1(x)\calg_1(y).
\end{align}
Similarly, for $(E;F)=(1\,3;2)$, and $(E;F)=(1\,2;3)$, the corresponding determinants are
\begin{align*}
     Q_{2}=\det[ J_{2}f]
     =-\calg_2(x)\calg_2(y),
\end{align*}
and
\begin{align}\label{eq: determinant_3var_Q3}
   Q_{3}= \det [J_{3}f]
   =-\calg_3(x)\calg_3(y),
\end{align}
respectively.

Let $\Phi$ be the restriction of $f$ to $U$, $\Phi=f\vert_U: U\to \R$. By \eqref{eq: restriction_f_U}, it follows immediately that $Q_{i_0}(x,y)\neq 0$, for all $(x,y)\in U\times U$, which implies that $\corank (J_{i_0}f)=0$ on $U\times U$. Without loss of generality, assuming that $i_0=1$, other cases can be treated similarly. The above analysis implies that $\Phi$ is $\Gamma_0(E;F)$-nondegenerate at every point $(x_0,y_0)\in Z$, where $(E;F)=(\,2\,3;1)$, and
\begin{align*}
    Z=\left\{ (x,y)\in U\times U: f(x)=f(y)\right\}.
\end{align*}
Applying Theorem \ref{thm: k_point_explicit_p=1} with $\Phi=f\vert_U$, $\alpha=n_{E,F}=3$, $m=0$, and $\left\{ A_i'\right\}_i$, we obtain the desired results. More precisely, in view of Theorem \ref{thm: k_point_explicit_p=1} (ii), one has $\calL^1(\Delta_\Phi(A_1',A_2',A_3'))>0$, provided
\begin{align*}
    \sum\limits_{i=1}^3\dim_HA_i'>\frac{\alpha}{2}+\frac{m}{2}+\frac{1}{2}=2.
\end{align*}
Using the fact that $A_i\supset A_i'$, and $\sum\limits_{i=1}^3\dim_HA_i'=\sum\limits_{i=1}^3\dim_HA_i$, one concludes that if 
\begin{align*}
    \sum\limits_{i=1}^3\dim_HA_i> 2,
\end{align*}
then $\calL^1\left(\Delta_f(A_1,A_2,A_3)\right)>0$, which gives part (i) of the theorem. 

Finally, if for some $i\in \left\{ 1,2,3\right\}$, one has $\nabla \calg_{i}(x)\neq 0$, for all $x\in \Omega$, then $\calZ_{3,f}$ is contained in a hypersurface in $\R^3$, hence $\dim \calZ_{3,f}\leq 2$. As a result, if $\dim_HA>2$, one has
\begin{align*}
    \dim_H(A\cap \calZ_{3,f}^c)=\dim_HA>2.
\end{align*}
Therefore, part (i) of the theorem follows from the case $A\subset \calZ_{3,f}^c$, which finishes the proof of the theorem.
\end{proof}
As indicated in the introduction, if $f\in C^\omega(\Omega,\R)$, then the set $\calZ_{3,f}$ is a zero set of a real analytic function, hence it has a better geometric structure. Exploiting this fact, we can determine the class of degenerate functions for Falconer-type problems are analytic special forms (see Definition \ref{def: analytic_special_form_3_vars}). Theorem \ref{thm: ElekesRonyai_functions_3vars_analytic_general} follows from the following theorem.
\begin{theorem}\label{thm: ElekesRonyai_functions_3vars_analytic_general_restate}
    Let $\Omega\subset \R^3$ be a connected open set, and let $f\in C^\omega(\Omega,\R)$. Suppose that $f$ is not an analytic special form and all its first order partial derivatives do not vanish identically on $\Omega$. 
    
    Let $A_i\subset \R$, $1\leq i\leq 3$, be compact sets with positive Hausdorff dimension such that $A_1 \times A_2\times A_3\subset  \Omega$.
    \begin{itemize}
        \item[(i)] For $u\in (0,1)$, if $\dim_HA_1+\dim_HA_2+\dim_HA_3>1+u$, then $\dim_H\Delta_f(A_1,A_2,A_3)\geq u$.
        \item[(ii)] If $\dim_HA_1+\dim_HA_2+\dim_HA_3>2$, then $\calL^1\left(\Delta_f(A_1,A_2,A_3)\right)>0$.
    \end{itemize}
\end{theorem}
Recall that an analytic function $f: \Omega\to \R$ is said to be an analytic special form if there exist a connected region $W\subset \Omega\setminus \bigcup\limits_{i=1}^3 \{\partial_{x_i}f=0\}$, and univariate real analytic functions $G,H_1,H_2$, and $H_3$ so that
    \begin{align}\label{eq: special_form_3vars}
        f(x_1,x_2,x_3)=G\big(H_1(x_1)+H_2(x_2)+H_3(x_3)\big),\quad \forall (x_1,x_2,x_3)\in W.
    \end{align}
    Before proving Theorem \ref{thm: ElekesRonyai_functions_3vars_analytic_general_restate}, we note that the auxiliary functions $\{\calg_i\}_{i=1}^3$ are useful for the following reason.
\begin{lemma}\label{lem: functions_3vars}
    Let $\Omega\subset \R^3$ be a connected open set and let $f\in C^\omega(\Omega)$ be analytic such that all its first order partial derivatives do not vanish identically on $\Omega$.  
 Then $\calg_i$ vanishes identically on $\Omega$ for all $1\leq i\leq 3$ if and only if $f$ is an analytic special form.
\end{lemma}
Assuming that Lemma \ref{lem: functions_3vars} is true for now, let us discuss the proof of Theorem \ref{thm: ElekesRonyai_functions_3vars_analytic_general_restate}.
\begin{proof}[Proof of Theorem \ref{thm: ElekesRonyai_functions_3vars_analytic_general_restate}]
    Assume $\Omega$ is a connected open set in $\R^3$ and $f\in C^\omega(\Omega,\R)$ is an analytic function satisfying the hypotheses of the theorem. We will prove part (i), part (ii) can be treated by the same argument. Let $A_1,A_2,A_3\subset \R$ be compact sets such that $A=A_1\times A_2\times A_3\subset \Omega$, and $\dim_H A>2$. 

    Recall that the functions $\left\{ \calg_i\right\}_{i=1}^3$ defined in \eqref{eq: def_gi_f} are real analytic functions, as $f$ is analytic, hence 
    $\calZ_{3,f}=\bigcap\limits_{i=1}^3 \calg_i^{-1}(0)$ is a real analytic subvariety of $\Omega$. If for some index $i$, $\calg_i\not\equiv 0$ on $\Omega$, then by Lemma \ref{lem: local_dim_away_analytic_variety}, for a fixed constant $0<\e< \dim_HA-2$, there exists a point $a=(a_1,a_2,a_3)\in \Omega\setminus \calZ_{3,f}$ such that for each index $i$, $A_i$ has local dimension $\geq \dim_HA_i-\e/4$. This means that there exists an open cube $Q=I_1\times I_2\times I_3\subset \Omega\setminus \calZ_{3,f}$ containing $a$ such that
    \begin{align*}
        \dim_H(A_i\cap I_i)\geq \dim_HA_i-\e/4, \quad \forall 1\leq i\leq 3.
    \end{align*}
    Putting $A_i'=A_i\cap I_i$ for each index $i$, we obtain that $A_1'\times A_2'\times A_3'\subset \Omega\setminus \calZ_{3,f}$ and\[\dim_HA_1'+\dim_HA_2'+\dim_HA_3'\geq \dim_HA_1+\dim_HA_2+\dim_HA_3-3\e/4>2.\]
    Applying Theorem \ref{thm: functions_3vars_smooth_restate} (i) with $A_1',A_2'$, and $A_3'$, one has $\calL^1(\Delta_f(A_1',A_2',A_3'))>0$, which implies $\calL^1\left(\Delta_f(A_1,A_2,A_3)\right)>0$ as desired.
    
    Now we assume that $\calg_i$ vanishes identially on $\Omega$, for $1\leq i\leq 3$. By Lemma \ref{lem: functions_3vars}, it follows that $f\vert_{W}$ is an analytic special form, contradicting the assumptions of the theorem. This completes the proof of the theorem.
\end{proof}

We now turn to the proof of Lemma \ref{lem: functions_3vars}.
\begin{proof}[Proof of Lemma \ref{lem: functions_3vars}]
Let $f: \Omega\to \R$ be an analytic function such that all its first order partial derivatives do not vanish identically on $\Omega$. If $f$ is an analytic special form, it is straightforward to check that $\calg_i$ vanishes identically on $\Omega$, for $1\leq i\leq 3$.

Now suppose that $\calg_i$ vanishes identically on $\Omega$, for $1\leq i\leq 3$, or equivalently,
\begin{align}
      (\partial_1f )(\partial_{23}f)-(\partial_{13}f)(\partial_2f)&=0\label{eq: pde_1},\\
     (\partial_3f)(\partial_{12}f)-(\partial_{13}f)(\partial_2f)&=0\label{eq: pde_2},\\
      (\partial_3f )(\partial_{12}f)-(\partial_{23}f)(\partial_1f)&=0\label{eq: pde_3},
\end{align}
for all $x\in \Omega$. We will show that $f$ is an analytic special form.

By the assumptions that $f$ is analytic and $\partial_if\not\equiv 0$, for all $1\leq i\leq 3$, one sees that the set
\begin{align*}
    W_f:=\left\{ x\in W: \partial_i f(x)=0 \text{ for some } i\in \left\{ 1,2,3\right\}\right\}
\end{align*}
is a union of submanifolds of dimension $\dim W_f\leq 2$. As a result, there exists an open set $W'\subset \Omega\setminus W_f$ such that $\partial_if(x)\neq 0$ for all $x\in W'$, for all $1\leq i\leq 3$.

Consider the restriction of the first equation \eqref{eq: pde_1} on $W'$, using the fact that $\partial_2f\neq 0$ on $W'$, we can write \eqref{eq: pde_1} as
\begin{align*}
    \frac{ (\partial_1f )(\partial_{23}f)-(\partial_{13}f)(\partial_2f)}{(\partial_2f)^2}=0, \quad \forall x\in W',
\end{align*}
which can be rewritten as
\begin{align*}
    \frac{\partial}{\partial x_3}\bigg(\frac{ \partial_1f }{\partial_2f}\bigg)=0, \quad \forall x\in W'.
\end{align*}
This yields that $\frac{ \partial_1f }{\partial_2f}$ does not depend on $x_3$ on $W'$. It follows that there exists an analytic function $a(x_1,x_2)\neq 0$ such that
\begin{align}\label{eq: partial_f_1f_2}
    \partial_1f=a(x_1,x_2)\partial_2f, \quad \forall x \in W'.
\end{align}
Similarly, from equations \eqref{eq: pde_2} and \eqref{eq: pde_3}, we claim that there exist analytic functions $b(x_2,x_3)\neq 0$, $c(x_1,x_3)\neq 0$ such that, for all $x\in W'$
\begin{align}
    \partial_2 f&=b(x_2,x_3)\partial_3f,\label{eq: partial_f_2f_3}\\
    \partial_1f&=c(x_1,x_3)\partial_3f. \label{eq: partial_f_1_f_3}
\end{align}
Plugging \eqref{eq: partial_f_2f_3} into  \eqref{eq: partial_f_1f_2}, we have the following
\begin{align*}
    \partial_1f=a(x_1,x_2)\partial_2f=a(x_1,x_2)b(x_2,x_3)\partial_3f,\quad \forall x\in W'.
\end{align*}
Comparing this with \eqref{eq: partial_f_1_f_3}, one finds that
\begin{align*}
    c(x_1,x_3)=a(x_1,x_2)b(x_2,x_3),\quad \forall x\in W'.
\end{align*}
Observe that the function $c(x_1,x_3)$ on the left hand side depends on variables $x_1,x_3$, and does not depend on $x_2$. On the right hand side, $a(x_1,\cdot)$ and $b(\cdot,x_3)$ depend on the variable $x_2$ on some nonempty interval $I$. Hence, there exists an open ball $\calB\subset W'$ such that $a$ and $b$ have the form
\begin{align*}
    a(x_1,x_2)=\tilde{a_1}(x_1)\tilde{a_2}(x_2),\quad b(x_2,x_3)=\frac{1}{\tilde{a_2}(x)}b_3(x_3), \quad \forall x\in \calB,
\end{align*}
where $\tilde{a_1},\tilde{a_2},b_3$ are nonzero analytic functions in one variable. Equations \eqref{eq: partial_f_1f_2} and \eqref{eq: partial_f_2f_3} become
\begin{align*}
    \partial_1 f=\tilde a_1(x_1)\tilde a_2(x_2)\partial_2f \quad\text{ and }\quad \partial_2f=\frac{1}{\tilde a_2(x_2)}b_3(x_3)\partial_3f,\quad \forall x\in \calB.
\end{align*}
These equations form a system of homogeneous first order partial differential equations, that is,
\begin{align*}
   \begin{cases}
       a_1(x_1)\partial_1f-a_2(x_2)\partial_2f=0 \\
       a_2(x_2)\partial_2f-a_3(x_3)\partial_3f=0
   \end{cases},\quad\quad  \forall x\in \calB,
\end{align*}
where $a_1(x_1)=\frac{1}{\tilde a_1(x_1)}$, $a_2(x_2)=\tilde a_2(x_2)$, and $a_3(x_3)=b_3(x_3)$ are nonzero analytic functions. Using the method of characteristics to solve the above system, one obtains the form of $f$. 

\begin{claim}
    There are analytic functions $G_0,H_1,H_2,H_3: \R\to \R$ and an open set $\tilde W\subset \calB$ such that
\begin{align}\label{eq: claim_sum_form}
    f(x_1,x_2,x_3)=G_0(H_1(x_1)+H_2(x_2)+H_3(x_3)), \quad \forall (x_1,x_2,x_3)\in \tilde W.
\end{align}
\end{claim}
  Without loss of generality, by translation, we assume that $\calB=B(0,r)$, with $r>0$. Let $L_1, L_2$ be the differential operators
\begin{align*}
    L_1=a_1(x_1)\partial_1 +a_2(x_2)\partial_2, \quad L_2=a_2(x_2)\partial_2+a_3(x_3)\partial_3,
\end{align*}
and consider the system of partial differential equations in $\calB$:
\begin{align}
    L_1(u)&=0,\quad \label{eq: pde_L1}\\
    L_2(u)&=0.\quad \label{eq: pde_L2}
\end{align}
For $s\in \R^3$, we also write $s=(s_1,s_2,s_3)$, with $s_i\in \R$, $1\leq i\leq 3$. Consider \eqref{eq: pde_L1}, and define a vector field $V_1$ on $\calB$ by
\begin{align*}
    V_1(s)=(a_1(s_1),a_2(s_2),0),\quad \forall s=(s_1,s_2,s_3)\in \calB.
\end{align*}
Observe that the surface $S=\left\{ (s_1,s_2,s_3)\in \calB : s_1=0\right\}$ is non-characteristic for $L_1$. Indeed, for any point $s\in S$, since $a_1(s_1)\neq 0$, $V_1(s)$ is not tangent to $S$ at $s$. Let $\phi: S\to \R$ be the restriction of $f$ on $S$, i.e.
\begin{align*}
    \phi(s)=f(s), \quad \forall s\in S.
\end{align*}
Invoking the existence and uniqueness theorem,  there exists a solution $u$ of \eqref{eq: pde_L1} with given initial values $u=\phi$ on $S$.
Next, we consider the characteristic curves of the equation \eqref{eq: pde_L1}, that is, the parametrized curves $x(t)$ that satisfy the system of ordinary differential equations
\begin{align*}
    \frac{dx_1}{dt}=a_1(x_1),\quad \frac{dx_2}{dt}=a_2(x_2),\quad \frac{dx_3}{dt}=0.
\end{align*}
It is easy to see that from the last equation, each characteristic curve $x(t)$ lies in the plane $H_c=\left\{ (x_1,x_2,x_3): x_3=c\right\}$ for some constant $c\in \R$. For each of these integral curves, one has
\begin{align*}
    \frac{dx_1}{dx_2}=\frac{a_1(x_1)}{a_2(x_2)},
\end{align*}
with $a_1$ and $a_2$ are nonzero and analytic. This yields that
\begin{align*}
    \frac{dx_1}{a_1(x_1)}-\frac{dx_2}{a_2(x_2)}=0.
\end{align*}
Integrating both sides, and put
$    G(x_1)=\int_0^{x_1}\frac{ds}{a_1(s)}$, $H(x_2)=\int_0^{x_2}\frac{d\eta}{a_2(\eta)}$, 
we obtain
\begin{align*}
    G(x_1)-H(x_2)=C,
\end{align*}
for some constant $C\in \R$. Since each point on the characteristic curve satisfying the above equation, for each $s=(0,s_2,s_3)\in S$, one has
\begin{align}\label{eq: relation_HG}
    G(0)-H(s_2)=C \quad \Longrightarrow \quad H(s_2)=-C=H(x_2)-G(x_1),
\end{align}
for all $(x_1,x_2, x_3)$ on the characteristic curve passing through $s$. 

Turn to equation \eqref{eq: pde_L2}, define a vector field $V_2$ on $\calB$ by
\begin{align*}
    V_2=(0,a_2(x_2),a_3(x_3)),\quad (x_1,x_2,x_3)\in \calB.
\end{align*}
As in the previous case, the surface $\tilde{S}=\left\{ (s_1,s_2,s_3)\in \calB: s_3=0\right\}$ is non-characteristic for $L_2$. Let $\tilde{\phi}:\tilde{S}\to \R$ be the restriction of $f$ on $\tilde{S}$. Then, by the existence and uniqueness theorem, there exists a solution $u$ to \eqref{eq: pde_L2} with given initial values $\tilde u=\tilde \phi$ on $\tilde S$. Consider the characteristic curves $\tilde{x}(\tilde t)$ with respect to \eqref{eq: pde_L2}, which satisfy
\begin{align*}
    \frac{d\tilde x_1}{d\tilde t}=0,\quad \frac{d\tilde x_2}{d\tilde t}=a_2(x_2),\quad \frac{d\tilde x_3}{d\tilde t}=\tilde a_3(x_3).
\end{align*}
Letting
\begin{align*}
     H(\tilde x_2)=\int_0^{\tilde x_2}\frac{d\eta}{a_2(\eta)},\quad K(\tilde x_3)=\int_0^{x_3}\frac{d\tilde\eta}{a_3(\tilde\eta)},
\end{align*}
and solving
\begin{align*}
    \frac{d\tilde x_2}{d\tilde x_3}=\frac{a_2(\tilde x_2)}{a_3(\tilde x_3)},
\end{align*}
 one sees that there exists $C'\in \R$ such that each characteristic curve satisfies
\begin{align*}
    H(\tilde x_2)-K(\tilde x_3)=C'.
\end{align*}
Thus, for each point $\tilde s=(\tilde s_1, \tilde s_2,0)\in \tilde S $, one has
\begin{align}\label{eq: relation_HK}
    H(\tilde s_2)=C' =H(\tilde x_2)-K(\tilde x_3) ,
\end{align}
for all $(\tilde x_1,\tilde x_2,\tilde x_3)$ on the characteristic curve passing through $\tilde s$.

Now we combine all the above calculations together to conclude that $f$ has the desired form. Recall that by the existence and uniqueness theorem, $f=u=\tilde u$ in a small neighborhood $\tilde W$ of the origin and is constant on characteristic curves $x(s,t)$ and $\tilde x(\tilde s,\tilde t)$ above. Choose a small interval $I\subset \R$ centers at the origin such that $\left\{ 0\right\}\times I\times \left\{ 0\right\}\subset S\cap \tilde S\cap \tilde W$. For $s_2\in I$, one has $s=(0,s_2,0)\in S$, and $f=u$ is invariant on the characteristic curve $x(s,t)$ passing through $s$, denoted by
\begin{align*}
    \gamma_{s_2}:=\left\{ (x_1(s,t),x_2(s,t),0): t \text{ small}\right\}.
\end{align*}
By the relation \eqref{eq: relation_HG}, one has
\begin{align}\label{eq: gamma_s_2}
    \gamma_{s_2}=\left\{ (x_1,\tilde s_2,0): H(\tilde s_2)-G(x_1)=H(s_2)\right\}\subset \tilde S\subset \R^2\times \left\{ 0\right\}.
\end{align}
Since this curve is contained in $\tilde S$, for each point $\tilde s=(x_1,\tilde s_2,0)\in \gamma_{s_2}$, the function $f=\tilde u$ is invariant on the characteristic curve $\tilde x(\tilde s,\tilde t)$ passing through $\tilde s$, denoted by
\begin{align*}
    \gamma_{s_2,\tilde s}:=\left\{ (x_1, \tilde x_2(\tilde s,\tilde t), \tilde x_3(\tilde s,\tilde t)): \tilde t \text{ small}\,\right\}\subset \left\{ x_1\right\}\times \R^2.
\end{align*}
Using \eqref{eq: relation_HK}, $\gamma_{s_2,\tilde s}$ can be determined by
\begin{align}\label{eq: gamma_s_2_stilde}
    \gamma_{s_2,\tilde s}=\left\{ (x_1,x_2,x_3)\in \tilde W: H(x_2)-K(x_3)=H(\tilde s_2)\right\}.
\end{align}
From \eqref{eq: gamma_s_2} and \eqref{eq: gamma_s_2_stilde}, it follows that for each $s_2\in I$, the function $f$ is invariant on
\begin{align*}
    \Lambda_{s_2}&:=\bigcup\limits_{\tilde s\in \gamma_{s_2}}\gamma_{s_2,\tilde s}\\
    &=\left\{  (x_1,x_2,x_3)\in \tilde W: H(x_2)-K(x_3)-G(x_1)=H(s_2)\right\}\\
    &=\left\{  (x_1,x_2,x_3)\in \tilde W: H^{-1}(H(x_2)-K(x_3)-G(x_1))=s_2\right\},
\end{align*}
where in the last line we used the fact that $H$ is invertible.
Define $\psi(s_2)=f(0,s_2,0)$, for $s_2\in I$, then $\psi$ is analytic. The discussion above yields that
\begin{align*}
    f(x_1,x_2,x_3)=\psi\big(H^{-1}(H(x_2)-K(x_3)-G(x_1))\big),\quad \forall (x_1,x_2,x_3)\in \tilde W.
\end{align*}
Putting $G_0=\psi\circ H^{-1}$, $H_1=-G$, $H_2=H$, and $H_3=-K$, we obtain that
\begin{align*}
    f(x_1,x_2,x_3)=G_0\big(H_1(x_1)+H_2(x_2)+H_3(x_3)\big),\quad \forall (x_1,x_2,x_3)\in \tilde W,
\end{align*}
as claimed in \eqref{eq: claim_sum_form}. This finishes the proof of the lemma.
\end{proof}

\section{Distances between general sets and sets lying on smooth hypersurfaces}\label{sec: distance_surface}
In this section, we present an application of our general theorems to a variant of the distance set problem. Given $A,B\subset \R^d$, the distance set between $A$ and $B$ is 
 \begin{align*}
    \Delta(A,B)=\left\{  \vert x-y\vert: x\in A,y\in B\right\}.
 \end{align*}
\begin{theorem}
    \label{thm: distance_surfaces}
    Let $d\geq 2$. Let $\calS\subset \R^d$ be a smooth hypersurface.  Suppose that $A\subset \R^d$ and $B\subset \calS$ are Borel sets.
    \begin{itemize}
        \item[(i)] For $u\in (0,1)$, if $\dim_HA+\dim_HB> d-1+u$, then $\dim_H(\Delta(A,B))\geq u$.
        \item[(ii)] If $\dim_HA+\dim_HB>d$, then $\calL^1(\Delta(A,B))>0$.
    \end{itemize}
    Here, $\calL^1$ denotes the $1-$dimensional Lebesgue measure.
\end{theorem}

\begin{remark}
\begin{itemize}
\item[] 
    \item[(i)] We note that when $\calS$ is a hyperplane and both $A,B\subset \calS$, the result follows from Falconer's theorem. 
    \item[(ii)] Theorem \ref{thm: distance_surfaces} is a generalization of the result in \cite{Pham_Falconerfunction_2025}, and the lower bound for part (ii) is sharp when $\calS$ is an affine hyperplane (see \cite{Pham_Falconerfunction_2025} for sharpness examples). However, it remains an open question whether the dimensional thresholds in the theorem are sharp for curved surfaces.
\end{itemize}
\end{remark}
Let $\calS\subset \R^d$, $d\geq 2$, be a smooth manifold. For $z\in \calS$, denote $T_z\calS$ as the tangent space of $\calS$ at $z$, and $T\calS$ as the tangent bundle of $\calS$. The union of all affine tangent spaces to $\calS$ is
\begin{align*}
    \tilde{T}\calS:= \bigcup\limits_{z\in S}(z+T_z\calS).
\end{align*}
Note that $\tilde{T}\calS\neq T\calS$. Additionally, if $\calS$ is a hyperplane, then $\tilde T\calS=\calS$.

We first show that the conclusions of Theorem \ref{thm: distance_surfaces} are true if the set $A$ is away from $\tilde{T}\calS$.
\begin{theorem}\label{thm: surface_distance_tangent_condition}
    Let $d\geq 2$, and let $\calS\subset \R^d$ be a smooth hypersurface. Let $A\subset \R^d\setminus \overline{\tilde T\calS}$ and $B\subset \calS$ be compact sets such that $\dim_HA,\dim_HB>0$. 
    \begin{itemize}
        \item[(i)] For $u\in (0,1)$, if $\dim_HA+\dim_HB>d-1+u$, then $\dim_H\Delta(A,B)\geq u$.
        \item[(ii)] If $\dim_HA+\dim_HB>d$, then $\calL^1(\Delta(A,B))>0$. 
    \end{itemize}
\end{theorem}

\begin{proof}
    We will prove part (ii), part (i) can be proved by a similar argument. In view of Theorem \ref{thm: 2_point_explicit_p=1}, we need to find a suitable smooth function $\Phi$ and show that $\Phi$ is $\Gamma_m$-nondegenerate, with $m\geq 0$. 
    
    Without loss of generality, assume that $\calS$ is a hypersurface parametrized by a smooth mapping $\psi: U\to \R^d$, where $U\subset \R^{d-1}$ is an open set and $\psi$ satisfies $\rank(D_u\psi)=d-1$ everywhere. Assume that $A\subset X= \R^d\setminus\overline{\tilde T\calS}$ and $B'\subset U$ are compact sets and put $B=\psi (B')$.
    
    Let $\Phi: X\times U\to \R$ be a smooth function defined by
    \begin{align*}
        \Phi(x,u):=\sum_{i=1}^d\big(x_i-\psi_i(u)\big)^2, \quad \forall (x,u)\in X\times U.
    \end{align*}
    Then $\Delta_\Phi(A,B')=\left\{\Phi(x,u): x\in A, u\in B'\right\}$. Since $(0,\infty)\mapsto (0,\infty)$, $r\mapsto \sqrt{r}$ is bi-Lipschitz on compact sets away from the origin, it suffices to show that $\calL^1(\Delta_\Phi(A,B'))>0$, which then implies $\calL^1(\Delta(A, B))>0$.
    
    In view of Theorem \ref{thm: 2_point_explicit_p=1}, for $(x,u), (x',u')\in X\times U$, the $2d\times 2d$ matrix $\calJ_\Phi$ (see \eqref{eq: matrix_J_2_points_explicit}) is given by
    \begin{align}\label{eq: matrix_J_distance}
        \calJ_\Phi=
        \begin{bmatrix}
    0&\nabla_{u}\Phi(x,u) &-\nabla_{x'}\Phi(x',u')\,\\
       (\nabla_{x}\Phi(x,u))^T &H_{x,u}\Phi(x,u)&0\,\\
       (\nabla_{u'}\Phi(x',u'))^T &0 &H_{u',x'}\Phi(x',u')\,
    \end{bmatrix}_{2d\times 2d}.
    \end{align}
    One calculates
    \begin{align*}
        \nabla_x\Phi=2(x-\psi(u)), \quad \nabla_u\Phi= 2(D_u\psi)^*(x-\psi(u)), \quad H_{x,u}\Phi=-2D_u\psi.
    \end{align*}
    Plugging this into \eqref{eq: matrix_J_distance}, one has
    \begin{align*}
        \calJ_\Phi=2\cdot\begin{bmatrix}
    0&(D_u\psi)^\ast(x-\psi(u)) &-(x'-\psi(u'))\,\\
       (x-\psi(u))^T &-D_u\psi&0\,\\
       [(D_{u'}\psi)^\ast(x'-\psi(u'))]^T &0 &-(D_{u'}\psi)^T\,
    \end{bmatrix}_{2d\times 2d}.
    \end{align*}
    The determinant $\det (\calJ_\Phi)$ of this matrix satisfies 
    \begin{align*}
        \left\vert \det (\calJ_\Phi)\right\vert=2\left\vert\bigg(\det
        \begin{bmatrix}
            (x-\psi(u))^T &-D_u\psi\,
        \end{bmatrix}_{d\times d} \cdot \det
        \begin{bmatrix}
            (x'-\psi(u'))^T &-D_{u'}\psi\,
        \end{bmatrix}_{d\times d}\bigg)\,\right\vert.
    \end{align*}
    Observe that for each $u\in U$, since rank $D_u\psi=d-1$, the determinant 
   \[\det
        \begin{bmatrix}
            (x-\psi(u))^T &-D_u\psi\,
        \end{bmatrix}_{d\times d}=0\] if and only if 
        \begin{align*}
            x-\psi(u)\in T_{\psi(u)}\calS=\left\{D_u\psi(\theta): \theta\in \R^{d-1}\right\},
        \end{align*}
    or equivalently, $x\in \psi(u)+T_{\psi(u)}\calS$. 
    
    By assumptions, for all $(x,u)\in X\times U$, one has $x\not\in \psi(u)+T_{\psi(u)}\calS$, and hence $\det(\calJ_\Phi)\neq 0$ everywhere on $X\times U\times X\times U$. This yields that $\corank(\calJ_\Phi)=0$, and $\Phi$ is $\Gamma_0-$nondegenerate in the sense of Theorem \ref{thm: 2_point_explicit_p=1}. Applying Theorem \ref{thm: 2_point_explicit_p=1} (ii) with $d_X=d$, $d_Y=d-1$ and $m=0$, we conclude that if 
    \begin{align*}
        \dim_HA+\dim_HB'> \frac{d+d-1}{2}+\frac{0}{2}+\frac{1}{2}=d,
    \end{align*}
    then $\calL^1(\Delta_\Phi(A,B'))>0$. This finishes the proof of the theorem.
\end{proof}
Note that if there exists a hyperplane $P$ such that $\dim_H(A\cap P)=\dim_HA$, then Theorem \ref{thm: distance_surfaces} follows from \cite{Pham_Falconerfunction_2025}. Thus, we can assume that $\dim_H(A\cap P)<\dim_HA$ for any hyperplane $P$ in $\R^d$. Theorem \ref{thm: distance_surfaces} will then follow from Theorem \ref{thm: surface_distance_tangent_condition} and the following lemma. Heuristically speaking, we can restrict $\calS$ to a sufficiently small cap so that the hypotheses of Theorem \ref{thm: surface_distance_tangent_condition} are satisfied. 
\begin{lemma}\label{lem: set_away_hyperplane}
    Let $d\geq 2$, $\alpha>0$, and let $\calS\subset \R^d$ be a smooth hypersurface.
    Assume that $A\subset \R^d$ and $B\subset \calS$ are compact sets with positive Hausdorff dimension and $\dim_HA+\dim_HB>\alpha$. Suppose that $\dim_H(A\cap P)<\dim_HA$ for any affine hyperplane $P$.  
    Then there exist open sets $U,V\subset \R^d$ and $\varepsilon>0$ such that $ U\subset \R^d\setminus \overline{\tilde T(\calS\cap V)}$,
    \begin{align*}
       \dim_H(A\cap U)\geq \dim_HA-\varepsilon, \quad  \dim_H(B\cap V)\geq\dim_HB-\varepsilon,
    \end{align*}   
    and $\dim_H(A\cap U)+\dim_H (B\cap V)>\alpha$.
\end{lemma}
\begin{proof}
    Let $\varepsilon>0$ be such that $0<\varepsilon<\min\{\dim_HA+\dim_HB-\alpha, \dim_HB,\dim_HA\}$. Since $\dim_HB>0$, by Lemma \ref{lem: local_dim}, there exists $z_0\in \calS$ such that 
    \begin{align*}
        \dim_H\left(B\cap B(z_0,r)\right)\geq \dim_HB-\e/4, \quad \forall r>0.
    \end{align*}
    Denote $T_{z_0}\calS$ as the tangent space to $\calS$ at $z_0$, and put $P_0=z_0+T_{z_0}\calS$. By assumptions, one has $\dim_H(A\cap P_0)<\dim_HA$. Thus, there exist $x_0\in A$ and $c>0$ such that 
    \begin{align*}
        \dim_H\left(A\cap B(x_0,r)\right)\geq\dim_HA -\e/4,\quad \forall r>0,
    \end{align*}
    and the distance from $x_0$ to $P_0$ is larger than $c$, namely $d(x_0,P_0)>c$. Denote $y_0\in P_0$ as the orthogonal projection of $x_0$ onto $P_0$. Choosing $0<r_1<\frac{c}{100}$ small enough such that for all $z\in V=B(z_0,r_1)\cap \calS$, one has $d(y_0,z+T_z\calS)<\frac{c}{100}$. It is easy to see that 
    \begin{align*}
        d\left(x_0, (z+T_z\calS)\right)>\frac{c}{2}, \quad \forall z\in B(z_0,r_1).
    \end{align*}
    Choosing $0<r_2<c/4$, and put $U=B(x_0,r_2)$. Then, by the above calculations, one has
    \begin{align*}
    x\not \in z+T_z\calS, \quad \forall x\in U, z\in V.
    \end{align*}
    This implies that $U\subset \R^d\setminus \overline{\tilde T(\calS\cap V)}$ as desired. Furthermore,
    \begin{align*}
        \dim_H(A\cap U)+\dim_H(B\cap V)\geq \dim_HA+\dim_HB-\e/2>\alpha.
    \end{align*}
    This finishes the proof of the lemma.    
\end{proof}

\section{Final comments}\label{sec: final_comments}
In the present paper, we focus only on the applications of Theorem \ref{thm: general_phi_configurations} for the case codimension $p=1$, as providing explicit nondegeneracy conditions for $p>1$ is more challenging. Obtaining improved dimensional thresholds for some specific geometric configuration problems of interest seems more plausible, and we hope to address this in a sequel. However, it is worth remarking that a deeper understanding of the canonical relations associated with the relevant FIOs is required, and finding optimal dimensional thresholds by optimizing over all possible partitions $\sigma\in \mathcal{P}_{2k}$ would be difficult. 

We observe that the dimensional threshold in Theorem \ref{thm: k_point_explicit_p=1} (ii) cannot be lowered in general, as our applications for the distance set problem and trivariate functions are sharp, namely Theorems \ref{thm: ElekesRonyai_functions_3vars_analytic_general} (ii) and \ref{thm: distance_surfaces} (ii). It would be interesting to know whether the dimension threshold in Theorem \ref{thm: Elekes_Ronyai_analytic_2vars} (ii) is sharp.
 
Finally, as indicated in the introduction, characterizing analytic special forms for $k-$variate analytic functions using our approach would be much more complicated. For instance, in Theorem \ref{thm: 2_point_explicit_p=1}, one may need to look at all possible nonempty subsets $E,F$ of $\{1,\dots, k\}$ and study $\rank(J_{E,F}\Phi)$ to characterize degenerate functions; when $k\geq 4$, this may be challenging.

\addtocontents{toc}{\protect\setcounter{tocdepth}{0}}
\section*{Acknowledgments}
This paper is part of the author’s Ph.D. thesis at the University of Rochester. The author would like to thank Alex Iosevich and Allan Greenleaf for their insightful discussions, encouragement, and constant support.

\addtocontents{toc}{\protect\setcounter{tocdepth}{1}}
\bibliographystyle{alpha}
\bibliography{ref} 
\end{document}